\documentclass{amsart}
 
\usepackage{mystyle}
\title[Towards $\pi_*L_{K(2)}Z$]{Towards the $K(2)$-local homotopy groups of $Z$}
\author[Prasit B.]{Prasit Bhattacharya$^{1,*}$}
\address{$^1$Department of Mathematics, University of Virginia, 131 Kerchof Hall, Charlottesville, VA 22904}
\address{$^1$Tel: +1(434)924-4919}
\address{$^*$Corresponding author}
\email{$^1$pb9wh@virginia.edu}
\author[P. Egger]{Philip Egger$^2$}
\address{$^{2}$Hummel Lab, Campus Biotech, Swiss Federal Institute of Technology Lausanne (EPFL), 1201 Geneva, Switzerland}
\address{$^2$Tel: +41(21)695-5040}
\email{$^2$philip.egger@epfl.ch}

\begin{document}

\maketitle 
\tableofcontents

\begin{abstract}
In \cite{BE}, we introduced a class $\ZZ$ of $2$-local finite spectra and showed that all spectra $Z\in\ZZ$ admit a $v_2$-self-map of periodicity $1$. The aim of this article is to compute the $K(2)$-local homotopy groups $\pi_*L_{K(2)}Z$ of all spectra $Z \in \ZZ$ using a homotopy fixed point spectral sequence, and we give an almost complete answer. The incompleteness lies in the fact that we are unable to eliminate one family of $d_3$-differentials and a few potential hidden $2$-extensions, though we conjecture that all these differentials and hidden extensions are trivial.


\end{abstract}
\begin{center}Keywords: $K(2)$-local, stable homotopy groups, Morava stabilizer group\end{center}
\section{Introduction} \label{Sec:intro}
We recently introduced (see~\cite{BE}) the class of all finite $2$-local type $2$ spectra $Z$ such that there is an isomorphism
\[ H_*Z \iso A(2) \modmod E(Q_2)\]
of $A(2)$-modules, where $A(2)$ is the subalgebra of the Steenrod algebra generated by  $Sq^1$, $Sq^2$ and $Sq^4$. We denote this class by $\ZZ$. Let $K(n)$ denote the height $n$ Morava $K$-theory and $k(n)$ its connective cover. Let $\tmf$ denote the connective spectrum of topological modular forms. The two key features of $\ZZ$ (see \cite{BE} for details) are as follows: 
\begin{itemize}
\item Every $Z \in \ZZ$ admits a self-map $v:\Sigma^6 Z\to Z$ which induces multiplication by $v_2^1$ on $K(2)_*$-homology of $Z$, i.e. $Z$ admits a $v_2^1$-self-map. 
\item Every $Z \in \ZZ$ satisfies $\tmf \sma Z \simeq k(2)$. 
\end{itemize}  
The purpose of this paper is to compute the $K(2)$-local homotopy groups of any $Z \in \ZZ$. 

It is difficult to overestimate the importance of $K(n)$-local computations in stable homotopy theory. At every prime $p$, the homotopy groups of $L_{K(1)} S^0$ have been known to capture the patterns in chromatic layer $1$ of the stable homotopy groups of spheres (also known as the image of $J$) since work of Adams \cite{Adams}. Likewise, the chromatic fracture square, the chromatic convergence theorem \cite{Ravnil}, as well as the nilpotence and periodicity theorems in \cite{HS}, suggest that the $K(n)$-local homotopy groups of $S^0$ or other finite spectra encapsulate information about the patterns in the $n$-th chromatic layer of the stable homotopy groups of spheres. 

However, our motivation to compute the $K(2)$-local homotopy groups of $Z$ comes from its relevance to the telescope conjecture due to Ravenel~\cite{Ravtel}. One of the various formulations of the telescope conjecture is as follows. Let $X$ be a $p$-local type $n$ spectrum. By \cite{HS}, $X$ admits a $v_n$-self-map $v:\Sigma^{t}X \to X$, i.e. a self-map such that $K(n)_*v$ is an isomorphism. Then the homotopy groups of the telescope of $X$  
\[ T(X) := \hocolim(X \overset{v}\to \Sigma^{-t}X \overset{v}\to \Sigma^{-2t}X \overset{v} \to \dots )\]
are the $v_n$-inverted homotopy groups of $X$, i.e. $\pi_*(T(X)) = v_n^{-1} \pi_*(X)$. Since $K(n)_* = \mathbb{F}_p[v_n^{\pm 1}]$, the localization of a spectrum with respect to $K(n)$ can be thought of as, roughly speaking, another way of `inverting $v_n$' in the homotopy groups of $X$. Moreover, there is always a natural map 
\[ \iota: T(X) \to L_{K(n)}X.\]

\begin{telconj}[Ravenel] \label{telescope}For every type $n$ spectrum $X$, the map $\iota$ is a weak equivalence.
\end{telconj}

It follows from the thick subcategory theorem \cite[Theorem~7]{HS}, that if the telescope conjecture is true for one $p$-local type $n$ finite spectrum then it is true for all $p$-local type $n$ finite spectra (see~\cite{Ravnil}). For chromatic height $n=1$, the telescope conjecture was proved by Haynes Miller \cite{Mil} using the mod $p$ Moore spectrum $M_p(1)$ when $p>2$, and by Mark Mahowald using the $bo$-resolution of the finite spectrum $Y:=M_2(1) \sma C\eta$ \cite{M1,M2} when $p=2$. While the telescope conjecture is true for $n\leq1$ at every prime, it remains an open question for all other pairs $(n,p)$.

We claim that in the case $n=2,p=2$, the $2$-local type $2$ spectra $Z \in \ZZ$ are the most appropriate ones to consider in our study of the telescope conjecture. Firstly, they all admit a $v_2^1$-self-map, whereas other type $2$ spectra with known $v_2$-periodicity, such as $M(1,4)$ and the $A_1$ spectra, only admit $v_2^{32}$-self-maps \cite{BHHM, BEM}. Lower periodicity is desirable for computational reasons. Moreover, the fact that $\tmf \sma Z \simeq k(2)$ makes the $E_1$-page of the $\tmf$-based Adams spectral sequence readily computable. Also, the $Z \in \ZZ$ are in many ways the `correct' height $2$ analogue of $Y$ (the spectrum used in the proof of the telescope conjecture at chromatic height $1$ at the prime $2$). This is because $Y$ is a type $1$ spectrum which satisfies properties analogous to $Z$, i.e. it admits a $v_1^1$-self-map \cite{DM81} and satisfies $bo \sma Y \simeq k(1)$. We will further strengthen our claim by giving an almost complete computation of the $K(2)$-local homotopy groups of any $Z \in \ZZ$, which is the `easier side of the telescope conjecture' because of its computational accessibility.

In this paper we will use a homotopy fixed point spectral sequence \eqref{eqn:homotopyfix}, which is essentially a consequence of the work of Jack Morava in \cite{Mor} followed by \cite{MRW} and \cite{DH2}. We will give further details in Section~\ref{Sec:FGL}.

To compute the homotopy fixed point spectral sequence, we need to understand the action of the big Morava stabilizer group $\Gt= \St \rtimes Gal(\mathbb{F}_4/\mathbb{F}_2)$ on $\E Z$, where $\St$ is the small Morava stabilizer group (see Section~\ref{Sec:FGL} for details). This action can be understood by explicitly analysing the $\BPBP$-comodule structure on $\BP Z$ via the map 
\[ \phi: \BPBP \to Hom^c(\St, \E Z)\] 
due to \cite{DH1}. The real hard work in this paper is to compute the $\BPBP$-comodule structure on $\BP Z$ and obtain the action of $\St$ on $\E Z$ via the map $\phi$. The group $\St$ has a finite quaternion subgroup $Q_8$ (to be described in Section~\ref{Sec:Morava}) and the pivotal result of this paper is Theorem~\ref{crucial}, where we prove that there is an isomorphism 
\[ \Ez Z\iso\Ff[Q_8]\]
of modules over the group ring $\Ff[Q_8]$. Part of the proof of Theorem~\ref{crucial} is a nontrivial exercise in representation theory, which we have banished to Lemma~\ref{appendixmain} in the appendix in order to avoid distracting from the main mathematical issues at hand. Theorem~\ref{crucial} provides another point of comparison between $Y$ and $Z$; note that $\Gn_1= \Z_2^{\times} \iso \Z/2 \times \Z_2$, and it can easily be seen that \[ (E_1)_0Y \iso \mathbb{F}_2[\Z/2].\]
In Section $5$, we run the algebraic duality resolution spectral sequence, a convenient tool to compute the group cohomology with coefficients in $\E Z$. Finally in Section~\ref{Sec:homotopy} we compute the the $E_2$-page of \eqref{eqn:homotopyfix}. We locate two possible families of $v_2$-linear $d_3$-differentials and several possible hidden extensions. Using the inclusion $S^0 \hookrightarrow Z $ of the bottom cell, we are able to eliminate one of the two $v_2$-linear $d_3$-differentials and some of the possible hidden extensions. 

\subsection*{Summary of results}
In Figure~\ref{LK2Z1}, we summarize all possibilities for $\pi_*L_{K(2)}Z$ from the work in this paper.  
Figure~\ref{LK2Z1} is a part of the homotopy fixed point spectral sequence, where we represent possible $d_3$-differentials using dashed arrows and hidden extensions by dotted lines. Any generator which is a multiple of a specific element $\zeta$ in the $E_2$-page (to be discussed in Section~\ref{Sec:homotopy}) is displayed using a `$\circ$', otherwise using a `$\bullet$'. 
\begin{figure}[!ht]
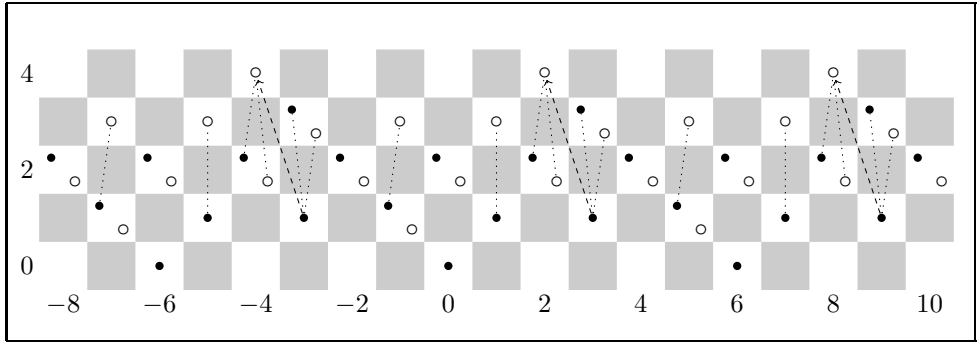

\[\begin{sseq}[grid=chess, entrysize= .64cm]{-8...10}{0...4}
\ssmoveto{-6}{0}\ssdropbull \ssmove{6}{0} \ssdropbull \ssname{g} \ssmove{6}{0} \ssdropbull
\ssmoveto{-7}{1} \ssdropbull \ssname{c} \ssdrop{\circ} \ssname{rho} \ssmove{6}{0} \ssdropbull \ssname{c1} \ssdrop{\circ} \ssname{rho1} \ssmove{6}{0} \ssdropbull \ssname{c2} \ssdrop{\circ} \ssname{rho2}
\ssmoveto{-5}{1} \ssdropbull \ssname{a} \ssmove{6}{0} \ssdropbull \ssname{a1} \ssmove{6}{0} \ssdropbull \ssname{a2} 
\ssmoveto{-3}{1} \ssdropbull \ssname{b} \ssmove{6}{0} \ssdropbull \ssname{b1} \ssmove{6}{0} \ssdropbull \ssname{b2}
\ssmoveto{-8}{2} \ssdropbull \ssname{ab} \ssdrop{\circ} \ssname{crho} \ssmove{6}{0} \ssdropbull \ssname{ab1} \ssdrop{\circ} \ssname{crho1}\ssmove{6}{0} \ssdropbull \ssname{ab2} \ssdrop{\circ} \ssname{crho2}\ssmove{6}{0} \ssdropbull \ssname{ab3} \ssdrop{\circ} \ssname{crho3} 
\ssmoveto{-6}{2} \ssdropbull \ssname{ac} \ssdrop{\circ}{arho} \ssmove{6}{0} \ssdropbull \ssname{ac1} \ssdrop{\circ}{arho1} \ssmove{6}{0} \ssdropbull \ssname{ac2} \ssdrop{\circ}{arho2}
\ssmoveto{-4}{2} \ssdropbull \ssname{bc} \ssdrop{\circ} \ssname{brho}\ssmove{6}{0} \ssdropbull \ssname{bc1} \ssdrop{\circ} \ssname{brho1}\ssmove{6}{0} \ssdropbull \ssname{bc2} \ssdrop{\circ} \ssname{brho2} \ssmove{6}{0} \ssdropbull \ssname{bc3} \ssdrop{\circ} \ssname{brho3}
\ssmoveto{-3}{3} \ssdropbull \ssname{abc} \ssdrop{\circ} \ssname{abrho} \ssmove{6}{0} \ssdropbull \ssname{abc1} \ssdrop{\circ} \ssname{abrho1} \ssmove{6}{0} \ssdropbull \ssname{abc2} \ssdrop{\circ} \ssname{abrho2}
\ssmoveto{-7}{3} \ssdrop{\circ} \ssname{acrho} \ssmove{6}{0} \ssdrop{\circ} \ssname{acrho1} \ssmove{6}{0} \ssdrop{\circ} \ssname{acrho2} \ssmove{6}{0}
\ssmoveto{-5}{3} \ssdrop{\circ} \ssname{bcrho} \ssmove{6}{0} \ssdrop{\circ} \ssname{bcrho1} \ssmove{6}{0} \ssdrop{\circ} \ssname{bcrho2} \ssmove{6}{0}
\ssmoveto{-4}{4} \ssdrop{\circ} \ssname{abcrho} \ssmove{6}{0} \ssdrop{\circ} \ssname{abcrho1} \ssmove{6}{0} \ssdrop{\circ} \ssname{abcrho2} \ssmove{6}{0}
\ssgoto{b} \ssgoto{abcrho} \ssstroke[dashed,arrowto]
\ssgoto{b1} \ssgoto{abcrho1} \ssstroke[dashed,arrowto]
\ssgoto{b2} \ssgoto{abcrho2} \ssstroke[dashed,arrowto]
\ssgoto{c} \ssgoto{acrho} \ssstroke[dotted]
\ssgoto{c1} \ssgoto{acrho1} \ssstroke[dotted]
\ssgoto{c2} \ssgoto{acrho2} \ssstroke[dotted]
\ssgoto{a} \ssgoto{bcrho} \ssstroke[dotted]
\ssgoto{a1} \ssgoto{bcrho1} \ssstroke[dotted]
\ssgoto{a2} \ssgoto{bcrho2} \ssstroke[dotted]
\ssgoto{b} \ssgoto{abc} \ssstroke[dotted]
\ssgoto{b} \ssgoto{abrho} \ssstroke[dotted]
\ssgoto{b1} \ssgoto{abc1} \ssstroke[dotted]
\ssgoto{b1} \ssgoto{abrho1} \ssstroke[dotted]
\ssgoto{b2} \ssgoto{abc2} \ssstroke[dotted]
\ssgoto{b2} \ssgoto{abrho2} \ssstroke[dotted]
\ssgoto{brho} \ssgoto{abcrho} \ssstroke[dotted]
\ssgoto{bc} \ssgoto{abcrho} \ssstroke[dotted]
\ssgoto{brho1} \ssgoto{abcrho1} \ssstroke[dotted]
\ssgoto{bc1} \ssgoto{abcrho1} \ssstroke[dotted]
\ssgoto{brho2} \ssgoto{abcrho2} \ssstroke[dotted]
\ssgoto{bc2} \ssgoto{abcrho2} \ssstroke[dotted]
\end{sseq}\]
\caption{Possible differentials and hidden extensions in the spectral sequence $H^s(\Gt;\Et Z)\To\pi_{t-s}L_{K(2)}Z$.}\label{LK2Z1}
\end{figure}
Since the homotopy groups of $L_{K(2)}Z$ are periodic with respect to multiplication by $v_2^1$, which has bidegree $(s,t-s) = (0,6)$, the different possible answers can be read off from the portion  $0 \leq t-s \leq 5$.

In work to appear, the $\tmf$-resolution for one particular model of $Z \in \ZZ$ is studied to compute its unlocalized homotopy groups. This computation shows that the potential $d_3$-differentials and hidden extensions as indicated in Figure~\ref{LK2Z1} are trivial, giving us a complete computation of the $K(2)$-local homotopy groups of that particular spectrum $Z$. We expect the same thing to happen for every spectrum $Z \in \ZZ$.
\begin{conj} \label{collapse} For every $Z \in \ZZ$, the $K(2)$-local homotopy groups of $Z$ are given by 
\[ \pi_*L_{K(2)}Z \iso \Ft[v_2^{\pm 1}]\otimes E(a_1, a_3, a_5, \zeta) \]
where $|a_i| = i$, $|\zeta| = -1$ and $|v_2| = 6$.
\end{conj}

The spectrum $Z$ in the unpublished work mentioned above would be the first finite $2$-local spectrum for which we have complete knowledge of its $K(2)$-local homotopy groups. It can be built using iterated cofiber sequences of five different self-maps (see \cite{BE}) starting from $S^0$. Thus, one could work backwards from $\pi_*L_{K(2)}Z$, using Bockstein spectral sequences iteratively to get information about $\pi_*L_{K(2)}S^0$.

\section*{Organization of the paper}  
The  results in this paper are independent of the choice of  $Z \in\ZZ$, and hence $Z$ will refer to an arbitrary spectrum $Z \in \ZZ$ for the rest of the paper. 

We devote Section~\ref{Sec:FGL} to recalling some fundamental results which connect the theory of formal group laws to homotopy theory. 

In Section~\ref{Sec:E2Z} we compute the $\BPBP$-comodule structure of $\BP Z$. 

In Section~\ref{Sec:Morava}, we briefly recall some of the details of the height $2$ Morava stabilizer group $\St$ and compute the action of $\St$ on the generators of $\E Z$.

In Section~\ref{Sec:Duality}, we compute the group cohomology with coefficients in $\E Z$ using the duality spectral sequence as well as a result of Henn, reported by Beaudry \cite{Bea2}. 

In Section~\ref{Sec:homotopy}, we analyse the homotopy fixed point spectral sequence for $Z$ and eliminate one of the two possible $\Ft[v_2^{\pm 1}]$-linear families of $d_3$-differentials and some of the possible hidden extensions.
     
In Appendix \ref{appendix}, we include the representation theory exercise omitted from the proof of Theorem~\ref{crucial}.
\section*{Acknowledgments}
The authors would like to thank Agn\`es Beaudry and Mark Behrens for their helpful comments and suggestions. We also thank Paul Goerss for his tireless work advising the second author's Ph.D. thesis, which was the starting point of this project. The idea of understanding eight dimensional $Q_8$-representations, as expressed in Appendix, stemmed from a discussion with Noah Snyder. We are also thankful to the referee for identifying some issues in an earlier draft and providing many helpful details and suggestions for improvement.
\section{Formal group laws and homotopy theory}\label{Sec:FGL}
The theory of formal group laws was developed by number theorists and eventually found by Lazard and Quillen to have deep relations with homotopy theory. We will review these relations, primarily following~\cite{LT} and~\cite{Rav}. We will conclude with a formula relating action of the Morava stabilizer group on a Morava module to the structure of a corresponding $\BPBP$-comodule.
\begin{defn}
Let $R$ be a $\Zp$-algebra. A \emph{formal group law} over $R$ is a power series $F(x,y)\in R[[x,y]]$ satisfying
\begin{itemize}
 \item $F(x,y)=F(y,x)$ 
 \item $x=F(x,0)$
 \item $F(F(x,y),z)=F(x,F(y,z))$
\end{itemize}
When $R$ is a graded $\Zp$-algebra we set $|x| = |y| = -2$ and we require that $F(x,y)$ be a homogeneous expression in degree $-2$.
\end{defn}
\begin{defn}
Given formal group laws $F,G$ over $R$, a \emph{homomorphism} from $F$ to $G$ is a power series $f\in R[[x]]$ such that $f(0)=0$ and $$f(F(x,y))=G(f(x),f(y)).$$ A homomorphism $f$ is an \emph{isomorphism} if $f'(0)$ is a unit in $R$, and an isomorphism $f$ is said to be \emph{strict} if $f'(0)=1$. A strict isomorphism from $F$ to the additive formal group law is called a \emph{logarithm} of $F$.
\end{defn}
\begin{notn}
We will often use the notation $x+_Fy$ to denote $F(x,y)$ and $[n]_F(x)$ to denote $\underbrace{x+_F\cdots+_Fx}_n$. We will denote the set of formal group laws over $R$ by $FGL(R)$, and the groupoid of formal group laws over $R$ with strict isomorphisms by $(FGL(R),SI(R))$. When $R$ is torsion-free, then the image of  $F$ in $(R \otimes \mathbb{Q})[[x,y]]$  has a logarithm, which we will denote by $\log_F  \in (R \otimes \mathbb{Q})[[x]]$
\end{notn}

\begin{defn}
Let $R$ be a torsion-free $\Zp$-algebra and let $F$ be a formal group law over $R$. Then $F$ is called \emph{$p$-typical} if its logarithm is\[\log_F(x)=\sum_{i\geq0}l_ix^{p^i}\]with $l_0=1$.
\end{defn}
Now we recall the $p$-local analogue of the famous theorem of Lazard and Quillen. All formal groups discussed will be assumed to be $p$-typical unless otherwise stated.

The assignment of a $\Zp$-algebra $R$ to the set $FGL(R)$ is functorial, and we denote this functor by
\[ FGL(-):  \Zp\textbf{-algebra} \to \textbf{Sets}.\]
Similarly, the functor which assigns a graded $\Zp$-algebra $R_*$ to the set of formal group laws over $R_*$ is denoted by 
\[ \overline{FGL}(-):  \textbf{Graded }\Zp\textbf{-algebra} \to \textbf{Sets}.\]
\begin{thm}[Cartier-Lazard-Quillen]The covariant functor $FGL(-)$ defined on the category of $\Zp$-algebras is represented by the $\Zp$-algebra \[ \tV
 = \Zp[\tilde{v}_1,\tilde{v}_2,\ldots],\]
 i.e. $FGL(R) \iso Hom_{\Zp}(\tV, R)$. The covariant functor $\overline{FGL}(-)$ defined on the category of graded $\Zp$-algebras is represented by the graded $\Zp$-algebra \[ \BP = \Zp[v_1,v_2,\ldots]\] with $|v_i| = 2(p^i-1)$, i.e. 
   $\overline{FGL}(R_*) \iso Hom_{\Zp}(\BP, R_* )$. 
\end{thm}
\begin{ex}[Honda formal group law] \label{ex:Honda}
Defining the ring homomorphism
\begin{eqnarray*}
\phi_n:\tV &\to&\Fpn\\
\tilde{v}_i&\mapsto&\left\{\begin{array}{rl}1&i=n\\0&i\neq n\end{array}\right.
\end{eqnarray*}
for $i\neq n$ gives the \emph{Honda formal group law} $\Gamma_n$ over $\Fpn$. This formal group law satisfies\[[p]_{\Gamma_n}(x)=x^{p^n}.\]
A theorem of Lazard says that $\Gamma_n$ is unique in that every formal group law of height $n$ over a separably closed field of characteristic $p$ is isomorphic to $\Gamma_n$, though this isomorphism might not be strict. 

\end{ex}

\begin{rem} \label{rem:defnvi}  The generators $\tilde{v}_i \in \tV$  are defined by the property that 
\[ [p]_{\mathcal{F}_{\tV}}(x) = px +_{\mathcal{F}_{\tV}} {\sum\limits_{i>1}}^{\mathcal{F}_{\tV}}\tilde{v}_ix^{p^i}\]
where $\mathcal{F}_{\tV}$ is the universal $p$-typical formal group law over $\tV$. 
Similarly, the $v_i \in \BP$ are defined by the property that 
\[  [p]_{\mathcal{F}_{\BP}}(\overline{x}) = p \overline{x} +_{\mathcal{F}_{\BP}} {\sum\limits_{i>0}}^{\mathcal{F}_{\BP}} v_i \overline{x}^{p^i}\]
where $\mathcal{F}_{\BP}$ is the universal $p$-typical formal group law over $\BP$ and $|\overline{x}| = -2$. The generators $\set{\tilde{v}_i: i>0}$ and $\set{v_i:i>0}$ are often called the \emph{Araki generators} in the literature. 
\end{rem}
Consider the functor 
\[ \rho: \Zp\textbf{-algebra}  \to \textbf{Graded }\Zp\textbf{-algebra}\]
which sends $R \mapsto R[u^{\pm1}]$, where $u$ is a formal variable in degree $-2$. If $F$ is a formal group law over $R$, then
\[ \overline{F}(\overline{x},\overline{y}) := uF(u^{-1}\overline{x}, u^{-1}\overline{y})\]
where $|\overline{x}| = |\overline{y}| = -2$, is a formal group law over $R[u^{\pm1}]$. Mapping $F \mapsto \overline{F}$ defines a natural transformation between the functors $FGL(-)$ and $\overline{FGL}(-) \circ \rho$.  Since $\overline{\mathcal{F}}_{\tV}$ is a formal group law over the graded ring $\tV[u^{\pm1}]$, we obtain a map 
\begin{equation} \label{theta1}
 \theta: \BP \to \tV[u^{\pm1}] 
 \end{equation}
and it follows from comparing the $p$-series (see Remark~\ref{rem:defnvi}) that $\theta(v_i) = u^{1- p^i}\tilde{v}_i$. 


We can also ask about how to represent \emph{groupoids} of formal group laws. We can do this in two ways, either by considering the groupoid of formal group laws with isomorphisms, or the smaller groupoid of formal group laws with strict isomorphisms.
\begin{lem}
Let $F$ be a $p$-typical formal group law and let $G$ be an arbitrary formal group law over a $\Zp$-algebra $R$, and let $f$ be an isomorphism from $F$ to $G$. Then $G$ is $p$-typical if and only if\[f^{-1}(x)=\sum\nolimits_{i\geq0}^F t_ix^{p^i},\]where $t_i\in R$ for every $i$ and $t_0\in R^{\times}$.
\end{lem}

If we want $f$ to be a strict isomorphism, then we must have $t_0=1$. In the context of graded $\Zp$-algebras, $t_i$ is forced to be in degree $2(p^i-1)$. Thus we can define a Hopf algebroid $(\BP,\BPBP)$ with\[\BPBP=\BP[t_1,t_2,\ldots:\vert t_i\vert=2(p^i-1)]\]which represents the functor 
$$(FGL(-),SI(-)): \textbf{Graded } \Zp\textbf{-algebras} \to \textbf{Groupoids}$$ 
which assigns a graded $\Zp$-algebra $R_*$ to the groupoid of $p$-typical formal group laws over $R_*$ with strict isomorphisms. Let $\eta_L, \eta_R: \BP \to \BPBP$ denote the left and the right units of the Hopf algebroid $\BPBP$. Note that the universal isomorphism $f: \eta_L^* \mathcal{F}_{\BP} = \mathcal{F}_{\BP} \to \eta_R^* \mathcal{F}_{\BP}$ satisfies the formula 
\[ f^{-1}(\overline{x}) = \overline{x} +_{\mathcal{F}_{\BP}} {\sum\limits_{i \geq 1}}^{\mathcal{F}_{\BP}} t_i\overline{x}^{p^i},\]
where $|\overline{x}| = -2$.

Similarly, one can consider the case where $R$ is ungraded and $f$ is an isomorphism that need not be strict. Thus we define
\[\tVT=\tV[\widetilde{t_0}^{\pm1},\widetilde{t_1},\widetilde{t_2},\ldots:\vert \widetilde{t_i}\vert=0],\]getting a Hopf algebroid $(\tV,\tVT)$ which represents the functor 
$$(FGL(-),I(-)): \Zp\textbf{-algebras} \to  \textbf{Groupoids}$$ which assigns a $\Zp$-algebra $R$ to the groupoid of $p$-typical formal group laws over $R$ with isomorphisms. In this case the universal isomorphism $\tilde{f}:\eta_L^* \mathcal{F}_{\tV} = \mathcal{F}_{\tV} \to \eta_R^*\mathcal{F}_{\tV}$ satisfies the formula 
\[ \tilde{f}^{-1}(x) = {\sum\limits_{i \geq 0}}^{\mathcal{F}_{\tV}} \widetilde{t_i}x^{p^i}.\]
Let us define
\begin{eqnarray*}
\overline{\mathcal{F}}_{\tV}(\overline{x},\overline{y})&=&u\mathcal{F}_{\tV}(u^{-1}\overline{x},u^{-1}\overline{y})\\
\hat{G}(\overline{x},\overline{y})&=&\widetilde{t_0}u \hspace{2pt} \eta_R^* \mathcal{F}_{\tV}(\widetilde{t_0}^{-1}u^{-1}\overline{x},\widetilde{t_0}^{-1}u^{-1}\overline{y}) \\
\hat{f}(\overline{x})&=& \widetilde{t_0}u\tilde{f}(u^{-1}\overline{x})
\end{eqnarray*}
where $|\overline{x}| = |\overline{y}| = -2$. It is easy to see that the triple $(\overline{\mathcal{F}}_{\tV}, \hat{f}, \hat{G})$ is an element of the groupoid $(FGL(\tVT[u^{\pm1}]), SI(\tVT[u^{\pm1}]))$. Hence the map $\theta$ of \eqref{theta1} can be extended to a left $\BP$-linear map 
\begin{equation} \label{theta2}
 \theta: \BPBP \to \tVT[u^{\pm1}].
 \end{equation}
Since
\begin{eqnarray*}
\hat{f}^{-1}(\overline{x}) &=&  u \tilde{f}^{-1}(\widetilde{t_0}^{-1}u^{-1} \overline{x})\\
  &=&  u({\sum\limits_{i \geq 0}}^{\mathcal{F}_{\tV}} \widetilde{t_i} \widetilde{t_0}^{-p^i}  u^{-p^i}\overline{x}^{p^i}) \\
  &=& {\sum\limits_{i \geq 0}}^{\overline{\mathcal{F}}_{\tV}} \widetilde{t_i} \widetilde{t_0}^{-p^i}  u^{1-p^i}\overline{x}^{p^i},
\end{eqnarray*}
and 
\[ \hat{f}^{-1}(\overline{x}) = \theta(\tilde{f}^{-1}(x)),\]
we get the formula 
\begin{equation} \label{thetati}
\theta(t_i) = \widetilde{t_i} \widetilde{t_0}^{-p^i}  u^{1-p^i}
\end{equation}

Now we briefly recall the notion of deformation, which arose in number theory, and has important implications for homotopy theory.
\begin{defn}
Let $k$ be a field of characteristic $p >0$ and $\Gamma$ a formal group law over $k$. A \emph{deformation} of $(k,\Gamma)$ to a complete local ring $B$ with projection $$\pi:B\to B/\mathfrak{m}$$ is a pair $(G,i)$ where $G$ is a formal group law over $B$ and $$i:k\to B/\mathfrak{m}$$ is a homomorphism satisfying $i\Gamma=\pi G$.
\end{defn} 
A morphism from $(G_1,i_1) \to (G_2,i_2)$ is defined only when $i_1=i_2$, in which case it consists of an isomorphism 
$$f:G_1\to G_2$$ of formal group laws over $B$ such that $$f(x)\equiv x \hspace{-.2em} \mod \mathfrak{m}.$$ Such morphisms are also called $\star$-isomorphisms. Note that the set $Def_{\Gamma}(B)$ of deformations of $(k,\Gamma)$ to $B$ with $\star$-isomorphisms forms a groupoid. The work of Lubin and Tate \cite{LT} guarantees the existence of a universal deformation. More precisely: 
\begin{thm}[Lubin-Tate] \label{thm:LT}
Let $\Gamma$ be a formal group law of finite height over a field $k$ of characteristic $p>0$. Then there exists a complete local ring $E(k,\Gamma)$ with residue field $k$ and a deformation $(F_{\Gamma},id)\in Def_{\Gamma}(E(k,\Gamma))$ such that for every $(G,i)\in Def_{\Gamma}(B)$, there is a unique continuous \footnote{A ring homomorphism of local ring is continuous if the image of the maximal ideal of the domain is contained in the maximal ideal of the codomain.} ring homomorphism  $\theta:E(k,\Gamma)\to B$ and a unique $\star$-isomorphism from $(G,i)$ to $(\theta F_{\Gamma}, i)$.
\end{thm}

\begin{rem}\label{EkGamma}
It is well-known (see~\cite{LT}) that if $k$ is a perfect field and $\Gamma$ has height $n$, then a choice of $F_{\Gamma}$ determines an isomorphism\[E(k,\Gamma)\cong W(k)[[u_1,\ldots,u_{n-1}]]\]of complete local rings, where $W(k)$ is the ring of Witt vectors on $k$. 
\end{rem}


The automorphism group $Aut(\Gamma/k)$ of $\Gamma$ acts on $E(k,\Gamma)$ as follows (also see \cite[\S1]{DH1}). Let $\gamma \in k[[x]]$ be an invertible power series.
Choose an invertible power series $\widetilde{\gamma} \in E(k,\Gamma)[[x]]$ as a lift of $\gamma$ and define $\widetilde{F_{\gamma}}$ over $E(k,\Gamma)$ by\[\widetilde{F_{\gamma}}(x,y):=\widetilde{\gamma}^{-1}(F_{\Gamma}(\widetilde{\gamma}(x),\widetilde{\gamma}(y))).\]
Note that the lift $\widetilde{F_{\gamma}}$ depends on the choice of lift $\widetilde{\gamma}$. Since 
$(\widetilde{F_{\gamma}},id)\in Def_{\Gamma}(E(k,\Gamma))$, the Lubin-Tate theorem gives us a unique homomorphism\[\widetilde{\phi}_{\gamma}:E(k,\Gamma)\to E(k,\Gamma)\] and a unique $\star$-isomorphism\[\hat{\gamma}: \widetilde{F_{\gamma}}\to\widetilde{\phi}_{\gamma}F_{\Gamma}.\]
The composite 
\[ f_{\gamma}: F_{\Gamma} \overset{\widetilde{\gamma}^{-1}}{\longrightarrow} \widetilde{F_{\gamma}} \overset{\hat{\gamma}}{\longrightarrow} \widetilde{\phi}_{\gamma}F_{\Gamma}\]
does not depend on the choice of $\widetilde{\gamma}$ and is an element of the groupoid $$(FGL(E(k,\Gamma)),I(E(k,\Gamma))).$$ 
Therefore the classifying map for $F_{\Gamma}$
\[ \widetilde{\theta_{\Gamma}}: \tV \to E(k,\Gamma),\]
can be extended to a left $\tV$-linear map 
\[ \widetilde{\theta_{\Gamma}}: \tVT \to Map^{c}(Aut(\Gamma/k),E(k,\Gamma)). \]
Let us simply denote $\theta_{\Gamma}(\widetilde{t_i})(\gamma)$ by $\widetilde{t_i}(\gamma)$ for $\gamma \in Aut(\Gamma/k)$.  The elements $\tilde{t_i}(\gamma)$ satisfy the equation 
\[ f_{\gamma}^{-1}(x) = {\sum \limits_{i\geq 0}}^{F_{\Gamma}} \widetilde{t_i}(\gamma) x^{p^i}. \]

One can also consider the graded formal group law $\overline{\Gamma}$ over $k[u^{\pm1}]$.  Note that $ Aut(\Gamma/ k) \iso Aut(\overline{\Gamma}/k[u^{\pm1}])$ via the invertible map $\gamma(-) \mapsto  u \gamma (u^{-1} -)$. One can similarly define the graded universal deformation formal group law $\overline{F}_{\Gamma}$ over the graded ring $E(k,\Gamma)[u^{\pm1}]$. Let $\gamma \in Aut(\Gamma/ k)$ act on $E(k,\Gamma)[u^{\pm1}]$ via the ring homomorphism 
$\phi_{\gamma}: E(k,\Gamma)[u^{\pm1}] \to E(k,\Gamma)[u^{\pm1}]$ such that 
\[ 
\phi_{\gamma}(x) = \left\lbrace \begin{array}{ccc}
\widetilde{\phi}_{\gamma}(x) & \text{ if $x \in E(k,\Gamma)$} \\ 
\widetilde{t_0}(\gamma) x & \text{ if $x = u$.} 
\end{array} \right.
\]
Notice that 
\[ \phi_{\gamma} \overline{F}_{\Gamma}(\overline{x}, \overline{y})  = \widetilde{t_0}(\gamma) u \widetilde{\phi}_{\gamma}F_{\Gamma}(\widetilde{t_0}(\gamma)^{-1}  u^{-1}\overline{x},  \widetilde{t_0}(\gamma)^{-1}u^{-1}\overline{y})  \]
and 
\[ \hat{f}_{\gamma} = \widetilde{t_0}(\gamma)^{-1} u f_{\gamma}(u^{-1} \overline{x})\]
is a strict isomorphism between $\overline{F}_{\Gamma}$ and  $\phi_{\gamma} \overline{F}_{\Gamma}$. Thus we have a left $\BP$-linear map 
\begin{equation} \label{map:theta}
 \theta_{\overline{\Gamma}}: \BPBP \to Map^c(Aut(\Gamma/k), E(k,\Gamma)[u^{\pm1}]). 
 \end{equation}
It can be easily checked that $\theta_{\overline{\Gamma}}$ is identical to the composite map 
\[ \BPBP \overset{\theta}\to \tVT[u^{\pm1}] \overset{ \widetilde{\theta_{\Gamma}}[u^{\pm1}]}{\longrightarrow} Map^c(Aut(\Gamma/k), E(k,\Gamma)[u^{\pm1}]). \]
Let us denote the map $\theta_{\overline{\Gamma}}(t_i)(-)$  simply by $t_i(-)$. It follows from \eqref{thetati} that 
\begin{equation} \label{titilde}
t_i(\gamma) = \widetilde{t_i}(\gamma) \widetilde{t_0}(\gamma)^{-p^i} u^{1-p^i}  
\end{equation}
for $\gamma \in Aut(\Gamma/k)$.  Also keep in mind that $f_{\gamma}$ fits into the commutative diagram 
\[ 
\xymatrix{
F_{\Gamma} \ar@{~>}[d]\ar[r]^{f_{\gamma}} & \widetilde{\phi}_{\gamma} F_{\Gamma} \ar@{~>}[d] \\
\Gamma \ar[r]_{\gamma} & \Gamma
}
\]
where the vertical squiggly arrows are reduction modulo $m = (p, u_1, \dots, u_{n-1}) $. Thus for $\gamma^{-1} = a_0x +_{\Gamma} a_1x^p +_{\Gamma} + a_2 x^{p^2}+_{\Gamma} \dots \in k[[x]]$, we have 
\begin{equation} \label{timodmax}
\widetilde{t_i}(\gamma) \equiv a_i \hspace{-8pt}\mod m \text{ and } t_i(\gamma) \equiv a_i a_0^{-p^i}u^{1-p^i} \hspace{-8pt}\mod m.
\end{equation}
It follows from \cite{Rav}[Corollary~$4.3.15$] that when $\Gamma$ has height $n$ and $k \leq n$,  
\begin{equation} \label{product-tk}
\widetilde{t_k}(\gamma_1 \gamma_2) \equiv {\sum\limits_{i=0}^k} \widetilde{t_i}(\gamma_1)\widetilde{t_{k-i}}(\gamma_2)^{p^{i}} \mod m.
\end{equation}


Now let's focus on the Honda formal group law $\Gamma_n$ over $\Fpn$ and let $F_n$ denote its universal deformation. By Remark~\ref{EkGamma}, we have\[E(\Fpn,\Gamma_n)=W(\Fpn)[[u_1,\ldots,u_{n-1}]],\]where $W(\Fpn)$ are the Witt vectors of $\Fpn$, which has an action of the \emph{small Morava stabilizer group}\[\Sn:=Aut(\Gamma_n/\Fpn).\]
Note that the map $\phi_n$ of Example~\ref{ex:Honda} which defines the Honda formal group law factors through $\mathbb{F}_p$. Therefore $\Gamma_n$ has coefficients in $\mathbb{F}_p$. Consequently, the \emph{big Morava stabilizer group}\[\Gn_n:=Aut(\Gamma_n/\Fp)=\Sn\rtimes Gal(\Fpn/\Fp)\] 
acts on $E(\Fpn,\Gamma_n)$. The Lubin-Tate universal formal group law $F_n$ over $E(\Fpn,\Gamma_n)$ is given by the ring homomorphism
\begin{eqnarray*}
\theta:\tV&\to&E(\Fpn,\Gamma_n) \\
\tilde{v}_i&\mapsto&\left\{\begin{array}{rl}u_i&i<n\\1&i=n\\0&i>n\end{array}\right.
\end{eqnarray*}
which means that 
\[ [p]_{\Gamma_n}(x) = p x +_{\Gamma_n} u_1 x^p +_{\Gamma_n} \dots +_{\Gamma_n} u_{n-1} x^{p^{n-1}} +_{\Gamma_n} x^{p^n} . \]
We also have a graded formal group law $\overline{F}_n$ over the graded ring $(E_n)_* :=  E(\Fpn,\Gamma_n)[u^{\pm1}]$ which is given by the ring homomorphism 
\begin{eqnarray*}
\theta:\BP&\to&(E_n)_* \\
v_i&\mapsto&\left\{\begin{array}{rl}u_iu^{-p^i}&i<n\\ u^{-p^n}&i=n\\0&i>n\end{array}\right.
\end{eqnarray*}
By the Landweber exact functor theorem,\[(E_n)_*(-):=(E_n)_*\otimes_{\BP}\BP(-)\]is a homology theory, thus it is represented by a spectrum $E_n$, known as \emph{Morava $E$-theory}. By a theorem of Devinatz and Hopkins (see~\cite{DH2}), the action of $\Gn_n$ on $(E_n)_*$ lifts to one on $E_n$ itself whose homotopy fixed point spectrum is \[ (E_n)^{h\Gn_n}\simeq L_{K(n)}S^0.\]
This gives us a homotopy fixed point spectral sequence
\begin{equation} \label{eqn:homotopyfix}
E_2^{s,t} := H^s(\Gn_n;(E_n)_t)\To\pi_{t-s}L_{K(n)}S^0
\end{equation}
whose $E_2$ page can be found using a Lyndon-Hochschild-Serre spectral sequence
\[H^{s_1}(Gal(\Fpn/\Fp);H^{s_2}(\Sn;(E_n)_*))\To H^{s_1+s_2}(\Gn_n;(E_n)_*),\]which as a consequence of Hilbert's Theorem~$90$ reduces to
\begin{equation}\label{Galois}H^s(\Gn_n;(E_n)_*)=H^s(\Sn;(E_n)_*)^{Gal(\Fpn/\Fp)}.\end{equation}

\section{The $\BPBP$-comodule $\BP Z$}\label{Sec:E2Z}
For every $Z \in \ZZ$, there is, by definition, an isomorphism$$H_*Z\iso (A(2) \modmod E(Q_2))_* = \Ft[\xi_1, \xi_2]/( \xi_1^8, \xi_2^4 )$$ 
of $A(2)_*$-comodules \cite{BE}. We will use this fact to determine the $\BPBP$-comodule structure of $\BP Z$. One can use the Adams spectral sequence 
\[ E_2^{s,t} = Ext^{s,t}_{A}(H^*BP \otimes H^*Z, \Ft) \Rightarrow BP_{t-s}Z\]
to compute $\BP Z$ as a $\BP $-module. Note that $$H^*BP = A \modmod E(Q_0, Q_1, Q_2, \dots)$$ where $Q_i$ are the Milnor primitives. By a change of rings, the $E_2$-page of the above Adams spectral sequence is isomorphic to 
\begin{equation} \label{ASSBPZ}
 E_2^{s,t} = Ext_{A}^{s,t}(H^*BP \otimes H^*Z, \Ft) \iso Ext^{s,t}_{E(Q_0,Q_1, \dots)}(H^*Z, \Ft). 
 \end{equation}
Let $g$ denote the generator of $H_*Z$ in degree $0$. As an $E(Q_0,Q_1,Q_2)$-module, \linebreak $A(2)\modmod E(Q_2)$ is a direct sum of $8$ copies of $E(Q_0,Q_1)$ generated by the elements in the set
\[\mathcal{G} = \lbrace \text{$g^*$, $(\xi_1^2g)^*$, $(\xi_1^4g)^*$, $(\xi_1^6g)^*$, $(\xi_2^2g)^*$, $(\xi_1^2\xi_2^2g)^*$, $(\xi_1^4\xi_2^2g)^*$, $(\xi_1^6\xi_2^2g)^*$} \rbrace. \]

Since $H^*Z \iso_{A(2)} A(2) \otimes_{E(Q_2)}\Ft$ and $Q_2$ is in the center of $A(2)$, $Q_2$ acts trivially on $H^*Z$. Using the iterative formula
\[ Q_i = Sq^{2^i}Q_{i-1} + Q_{i-1}Sq^{2^i}\]
one can inductively argue that $Q_i$ for $i \geq 2$ acts trivially on $H^*Z$. Thus, we have completely determined $H^*Z$ as a module over $E(Q_0, Q_1, \dots)$ from its $A(2)$-module structure. Thus as an $E(Q_0, Q_1, \dots)$-module 
\[ H^*Z \iso E(Q_0, Q_1, \dots ) \otimes_{E(Q_2, Q_3, \dots) }\mathcal{G} \]
and therefore the $E_2$-page of  \eqref{ASSBPZ} is isomorphic to 
\[ E_2^{s,t} \iso \Ft[v_2, v_3, \dots ] \otimes \mathcal{G}^* \]
where $v_i$ has bidegree $(s,t) = (1, |Q_i|) = (1, 2^{i+1} -1)$. Due to sparseness, the Adams spectral sequence \eqref{ASSBPZ} collapses at the $E_2$ page. Hence, as a $\BP $-module
\begin{equation} \label{eqn:BP2Z}
\BP Z \iso \BP /( 2, v_1)\langle x_0, x_2,x_4,x_6,y_6,y_8,y_{10},y_{12} \rangle, 
\end{equation}
 where $x_i$ and $y_i$ are generators in degree $i$ chosen in such a way that the map $\BP Z \to H_*Z$ sends:
\[ \begin{array}{llll}
x_0 \mapsto g  && y_6\mapsto \xi_2^2g\\
x_2 \mapsto \xi_1^2g && y_8 \mapsto\xi_1^2 \xi_2^2g\\
x_4 \mapsto \xi_1^4g && y_{10} \mapsto\xi_1^4 \xi_2^2g\\ 
x_6 \mapsto \xi_1^6g && y_{12} \mapsto\xi_1^6 \xi_2^2g.
\end{array} \]
This identification allows us to infer the $\BPBP$-comodule structure of $\BP Z$ from the $A(2)_*$-comodule structure of $H_*Z$ via the diagram 
\[ 
\xymatrix{
\BP Z \ar[r]^-{\psi} \ar[d]  & \BPBP \otimes_{\BP} \BP Z \ar[d]\\
 H_*Z \ar[r]_{\psi_2} & A(2)_* \otimes H_*Z
}
\]
 First notice that the co-action map 
\[\psi_2: H_*Z \to A(2)_* \otimes H_*Z\]sends 
\begin{equation} \label{eqn:coopA2Z}
\begin{array}{llll}
g &\mapsto& 1|g \\
\xi_1^2g &\mapsto& \xi_1^2|g + 1| \xi_1^2g \\
\xi_1^4g &\mapsto& \xi_1^4|g + 1|\xi_1^4g \\
\xi_1^6g &\mapsto& \xi_1^6|g + \xi_1^4|\xi_1^2g + \xi_1^2|\xi_1^4g + 1| \xi_1^6g \\
\xi_2^2g &\mapsto& \xi_2^2|g + \xi_1^4|\xi_1^2g + 1|\xi_2^2g \\
\xi_1^2\xi_2^2g &\mapsto& \xi_1^2\xi_2^2|g + (\xi_1^6+ \xi_2^2)|\xi_1^2g + \xi_1^2|\xi_2^2g + \xi_1^4|\xi_1^4g + 1|\xi_1^2\xi_2^2g \\ 
\xi_1^4\xi_2^2g &\mapsto& \xi_1^4\xi_2^2|g + \xi_1^8|\xi_1^2g + \xi_1^4|\xi_2^2g + \xi_2^2|\xi_1^4g + \xi_1^4|\xi_1^6g + 1|\xi_1^4\xi_2^2g \\
\xi_1^6\xi_2^2g &\mapsto&  \xi_1^6\xi_2^2|g + (\xi_1^4\xi_2^2 + \xi_1^{10})|\xi_1^2g + (\xi_1^2\xi_2^2 + \xi_1^8)|\xi_1^4g +\xi_1^6| \xi_2^2g\\
&& + (\xi_2^2 + \xi_1^6)|\xi_1^6g  + \xi_1^4|\xi_1^2\xi_2^2g + \xi_1^2|\xi_1^4\xi_2^2g + 1|\xi_1^6\xi_2^2g   
\end{array}   
\end{equation}
The map 
\[ \BPBP \longrightarrow A_* \]
sends $v_i \mapsto 0$ and $t_i \mapsto \zeta_i^2$, where $\zeta_i$ is the image of $\xi_i$ under the canonical antiautomorphism of $A_*$. Moreover $A(2)_* \iso A_*/(\zeta_1^8, \zeta_2^4, \zeta_3^2, \zeta_4, \zeta_5, \dots)$. Therefore $\psi_2$,
along with the fact that  $(2,v_1) \subset \BP$ acts trivially on $\BP Z$, completely  determines the composite map 
\[ \BP Z \overset{\psi}{\to} \BPBP \otimes_{\BP} \BP Z \to \BPBP/\mathcal{I}_2 \otimes_{\BP} \BP Z \]  
where \[\mathcal{I}_2 = (v_2,v_3, \dots, t_1^4, t_2^2, t_3, t_4, \dots)  \subset \BPBP. \]  
Note that all elements in the generating set $\set{x_0, x_2, x_4, x_6, y_6, y_8, y_{10}, y_{12}}$ of $\BP Z$  have internal degrees between $0$ and $12$, whereas $|t_j|> 12$ and $|v_j|>12$ when $j\geq 3$. Therefore, for $j \geq 3$, $t_j$ and $v_j$ do not appear in the expression for $\psi(x_i)$ and $\psi(y_i)$, though $v_2$ may be present. Using \eqref{eqn:coopA2Z} and the fact that $\zeta_1^2= \xi_1^2$ and $\zeta_2^2 = \xi_2^2 + \xi_1^6$, we easily derive the coaction map $\psi$ on the generators of $\BP Z$ modulo $(v_2,t_1^4,t_2^2) \in \BPBP$. We get: 
\begin{equation} \label{eqn:coopmodv2}
\begin{array}{lll}
\psi(x_0) &=& 1|x_0 \\
\psi(x_2) &=& t_1|x_0 + 1|x_2 \\
\psi(x_4) &=& t_1^2|x_0 + 1|x_4 \\
\psi(x_6) &\equiv& t_1^3|x_0 + t_1^2|x_2 + t_1|x_4 + 1|x_6 \\
\psi(y_6) &\equiv& (t_2 + t_1^3)|x_0 + t_1^2|x_2 + 1|y_6 \\
\psi(y_8) &\equiv&   t_1t_2|x_0 + t_2|x_2 + t_1^2|x_4 + t_1|y_6 + 1|y_8 \\
\psi(y_{10}) &\equiv& t_1^2t_2|x_0  + (t_1^3+t_2)|x_4 + t_1^2|x_6 + t_1^2|y_6 + 1|y_{10}\\
\psi(y_{12}) &\equiv&  t_1^3t_2|x_0 + t_1^2t_2|x_2 + t_1t_2|x_4 + t_2|x_6 + t_1^3|y_6 \\
&& + t_1^2|y_8 + t_1|y_{10} + 1|y_{12}  
\end{array}
\end{equation}
\begin{lem} \label{lem:independant} For any $Z \in \ZZ$, $\BP Z$ has one of the four different $\BPBP$-comodule structures given below:
\begin{equation} \label{eqn:coopfour}
\begin{array}{lll}
\psi(x_0) &=& 1|x_0 \\
\psi(x_2) &=& t_1|x_0 + 1|x_2 \\
\psi(x_4) &=& t_1^2|x_0 + 1|x_4 \\
\psi(x_6) &=& t_1^3|x_0 + t_1^2|x_2 + t_1|x_4 + 1|x_6 \\
\psi(y_6) &=& (t_2 + t_1^3)|x_0 + t_1^2|x_2 + 1|y_6 \\
\psi(y_8) &=&  (a t_1^4 + t_1t_2)|x_0 + t_2|x_2 + t_1^2|x_4 + t_1|y_6 + 1|y_8 \\
\psi(y_{10}) &=& (t_1^5 + t_1^2t_2)|x_0 + t_1^4|x_2 + (t_1^3+t_2)|x_4 + t_1^2|x_6 + t_1^2|y_6 + 1|y_{10}\\
\psi(y_{12}) &=& ((b+1) t_1^6 + t_1^3t_2 + (a+b)t_2^2)|x_0 + t_1^2t_2|x_2 + (bt_1^4 +t_1t_2)|x_4 + t_2|x_6 + t_1^3|y_6 \\
&& + t_1^2|y_8 + t_1|y_{10} + 1|y_{12}  
\end{array}
\end{equation}
where $a,b \in \Ft$. 
\end{lem}
 \begin{proof} 
For degree reasons, there are coefficients\[\lambda_6^0,\kappa_6^0,\mu_8^0,\lambda_8^0,\lambda_8^2,\mu_{10}^0,\lambda_{10}^0,\mu_{10}^2,\lambda_{10}^2,\lambda_{10}^4,\mu_{12}^0,\nu_{12}^0,\lambda_{12}^0,\kappa_{12}^0,\sigma,\mu_{12}^2,\lambda_{12}^2,\mu_{12}^4,\lambda_{12}^4,\lambda_{12}^6,\kappa_{12}^6\in\mathbb{F}_2\]such that one has
\begin{eqnarray*}
\psi(x_0)&=&1|x_0\\
\psi(x_2)&=&t_1|x_0+1|x_2\\
\psi(x_4)&=&t_1^2|x_0+1|x_4\\
\psi(x_6)&=&(t_1^3+\lambda_6^0v_2)|x_0+t_1^2|x_2+t_1|x_4+1|x_6\\
\psi(y_6)&=&(t_1^3+t_2+\kappa_6^0v_2)|x_0+t_1^2|x_2+1|y_6\\
\psi(y_8)&=&(\mu_8^0t_1^4+t_1t_2+\lambda_8^0v_2t_1)|x_0+(t_2+\lambda_8^2v_2)|x_2+t_1^2|x_4+t_1|y_6+1|y_8\\
\psi(y_{10})&=&(\mu_{10}^0t_1^5+t_1^2t_2+\lambda_{10}^0v_2t_1^2)|x_0+(\mu_{10}^2t_1^4+\lambda_{10}^2v_2t_1)|x_2+\\
&&+(t_1^3+t_2+\lambda_{10}^4v_2)|x_4+t_1^2|x_6+t_1^2|y_6+1|y_{10}\\
\psi(y_{12})&=&(\mu_{12}^0t_1^6+t_1^3t_2+\nu_{12}^0t_2^2+\lambda_{12}^0v_2t_1^3+\kappa_{12}^0v_2t_2+\sigma v_2^2)|x_0+\\
&&+(\mu_{12}^2t_1^5+t_1^2t_2+\lambda_{12}^2v_2t_1^2)|x_2+(\mu_{12}^4t_1^4+t_1t_2+\lambda_{12}^4v_2t_1)|x_4+\\
&&+(t_2+\lambda_{12}^6v_2)|x_6+(t_1^3+\kappa_{12}^6v_2)|y_6+t_1^2|y_8+t_1|y_{10}+1|y_{12}.
\end{eqnarray*}
The  counitality condition of $\psi$
\begin{equation} \label{eqn:counit}
\xymatrix{
\BP Z  \ar[drr]^{\cong} \ar[d]_{\psi}\\
\BPBP \otimes_{\BP} \BP Z \ar[rr]_-{\epsilon \otimes \BP Z} && \BP \otimes_{\BP} \BP Z 
}
\end{equation}
forces $\lambda_6^0=\kappa_6^0=\lambda_8^2=\lambda_{10}^4=\sigma=\lambda_{12}^6=\kappa_{12}^6=0.$ After a change of basis change 
\begin{eqnarray*}
y_8&\rightsquigarrow&y_8+\lambda_8^0v_2x_2\\
y_{10}&\rightsquigarrow&y_{10}+\lambda_{10}^0v_2x_4\\
y_{12}&\rightsquigarrow&y_{12}+\kappa_{12}^0v_2y_6+(\lambda_{12}^0+\kappa_{12}^0)v_2x_6.
\end{eqnarray*}
we have
\begin{eqnarray*}
\psi(x_0)&=&1|x_0\\
\psi(x_2)&=&t_1|x_0+1|x_2\\
\psi(x_4)&=&t_1^2|x_0+1|x_4\\
\psi(x_6)&=&t_1^3|x_0+t_1^2|x_2+t_1|x_4+1|x_6\\
\psi(y_6)&=&(t_1^3+t_2)|x_0+t_1^2|x_2+1|y_6\\
\psi(y_8)&=&(\mu_8^0t_1^4+t_1t_2)|x_0+t_2|x_2+t_1^2|x_4+t_1|y_6+1|y_8\\
\psi(y_{10})&=&(\mu_{10}^0t_1^5+t_1^2t_2)|x_0+(\mu_{10}^2t_1^4+\lambda_{10}^2v_2t_1)|x_2+(t_1^3+t_2)|x_4+t_1^2|x_6+t_1^2|y_6+1|y_{10},\\
\psi(y_{12})&=&(\mu_{12}^0t_1^6+t_1^3t_2+\nu_{12}^0t_2^2)|x_0+(\mu_{12}^2t_1^5+t_1^2t_2+(\lambda_8^0+\lambda_{12}^0+\lambda_{12}^2)v_2t_1^2)|x_2+\\
&&(\mu_{12}^4t_1^4+t_1t_2+(\lambda_{10}^0+\lambda_{12}^4+\lambda_{12}^0+\kappa_{12}^0)v_2t_1)|x_4+t_2|x_6+t_1^3|y_6+t_1^2|y_8+t_1|y_{10}+1|y_{12},\\
\end{eqnarray*}
Now we exploit the coassociativity condition
\begin{equation} \label{eqn:coass}
\xymatrix{
\BP Z  \ar[rr]^{\psi} \ar[d]_{\psi} && \BPBP \otimes_{\BP} \BP Z \ar[d]^{\Delta \otimes \BP Z}\\
\BPBP \otimes_{\BP} \BP Z \ar[rr]_-{ \BPBP \otimes \psi} && \BPBP \otimes_{\BP} \BPBP \otimes_{\BP} \BP Z.
}
\end{equation}
Applying the coassociativity condition on on $y_8$ tells us nothing, while applying it on $y_{10}$ tells us that
 \[\lambda_{10}^2 = 0, \mu_{10}^0 = \mu_{10}^2 = 1.\]
Applying it on $y_{12} $, we get 
\[ \lambda_8^0+\lambda_{12}^0+\lambda_{12}^2 = 0, \mu_{12}^2 = 0, \lambda_{10}^0+\lambda_{12}^4+\lambda_{12}^0+\kappa_{12}^0 = 0, \mu_{12}^4 + \mu_{12}^0+1=0 \text{ and } \mu_8^0 + \mu_{12}^0 + \nu_{12}^0+1=0 \]
Setting $a = \mu_{8}^0$ and $b = \mu_{12}^0+1$, we get \eqref{eqn:coopfour}. 
\end{proof} 
\begin{rem}By sending $v_i \mapsto 0$ and $t_i \mapsto \zeta_i^2$, we obtain a functor 
\[ Q : ( \BP, \BPBP)\text{-comodules}\longrightarrow ( \Ft, \Phi A_*)\text{-comodules},\] 
where $\Phi A_*$ is the double of the dual Steenrod algebra. This functor sends $\BP Z$ to $\Phi A(1)_*$. Since $A(1)_*$ has four different $A_*$-comodule structures, it follows that $\Phi A(1)_*$ has four different $\Phi A_*$-comodule structures. The four different $\BPBP$-comodule structures on $\BP Z$ are essentially lifts of the four different $\Phi A_*$-comodule structures on $ \Phi A(1)_*$.
\end{rem}
\begin{rem} \label{rem:M}Let $M_* = \BP/(2,v_1)\langle g_0, g_2, g_4, g_6 \rangle$ be the $\BPBP$-comodule with four generators with cooperations 
\begin{eqnarray*}
\psi(g_0)&=&1|g_0\\
\psi(g_2)&=&t_1|g_0+1|g_2\\
\psi(g_4)&=&t_1^2|g_0+1|g_4\\
\psi(g_6)&=&t_1^3|g_0+t_1^2|g_2+t_1|g_4+1|g_6.
\end{eqnarray*}
Then if $W = A_1 \sma C\nu$, where $A_1$ is any of the four $8$-cell complexes whose cohomology is isomorphic to $A(1)$, the $\BPBP$-comodule $\BP W$ is isomorphic to $M_*$.
\end{rem}
A straightforward calculation tells us: 
\begin{lem} \label{BPexactZ} There is an exact sequence of $\BPBP$-comodules 
\[ 0 \to M_* \overset{\iota}\to \BP Z \overset{\tau}\to \Sigma^6 M_* \to 0\]
such that $\iota(g_i) = x_i$, $\tau(x_i) = 0$ and $\tau(y_i) = \Sigma^6g_{i-6}$.
\end{lem}

\section{The action of the small Morava stabilizer group on $\E Z$}\label{Sec:Morava}
To compute the $E_2$-page of the homotopy fixed point spectral sequence
\begin{equation} \label{eqn:descent}
 E_2^{s,t} = H^s(\St;\Et Z)^{Gal(\Ff/\Ft)}\To\pi_{t-s}L_{K(2)}Z,
 \end{equation}
we first need to understand the action of $\St=Aut(\Gamma_2/\Ff)$ on $\E Z$, where $\Gamma_2$ is the height $2$ Honda formal group law over $\mathbb{F}_4$.  Recall from Section~\ref{Sec:FGL} \eqref{map:theta}
the left $\BP$-linear map 
\[ \theta_{\overline{\Gamma}_n}: \BPBP \to Map^c(\St,(E_2)_*).\]
 For $X$ a finite spectrum, we deduce the action of $\St$ on $\E X$ from the knowledge of the $\BPBP$-coaction map $\psi_{X}^{BP}$ on $\BP X$ 
 via the diagram
\begin{equation}\label{Zaction}
\xymatrix{
\BP X\ar[r]^-{\psi_X^{BP}}& \BPBP\otimes_{\BP}\BP X\ar[d]^{\theta_{\overline{\Gamma}_n}\otimes\BP X}\\
&Map(\St,(E_2)_*)\otimes_{\BP}\BP X.}
\end{equation}
The main purpose of this section is to understand the action of $\St$ on $\E Z$, for all $Z \in \widetilde{\mathcal{Z}}$. 

We begin by briefly recalling some key facts about $\St$ that we need for the calculations to follow. Let $T$ be a formal variable that need not commute with $W(\Ff)$ and let 
\[ \mathcal{O}_2 := W(\Ff)\langle T\rangle/(T^2 - 2, T\omega - \omega^\sigma T),\]
where $\omega$ is a root of $x^2 + x +1 \in \Ft[x]$, and $\sigma$ is the Frobenius map on $W(\Ff)$. Note that any element $\gamma \in \mathcal{O}_2$ can be written as a power series\[\gamma = \sum_{n=0}^{\infty}a_nT^n\] 
where the $a_n$ are Teichm\"uller lifts of $\Ff^{\times}$ or are zero. Then $\gamma$ corresponds to the power series  
\[ \overline{a}_0x +_{\Gamma_2} \overline{a}_1x^2 +_{\Gamma_2} \dots +_{\Gamma_2} \overline{a}_nx^{2^n} +_{\Gamma_2} \dots \in \Ff[[x]]\]
where $ \overline{a}_n$ is the image of $a_n$ under the quotient map $W(\Ff) \to \Ff$. In fact, this defines an isomorphism from $\mathcal{O}_2$  to $\mathit{End}(\Gamma_2/\Ff) \subset \Ff[[x]]$ and consequently, $\St$ is isomorphic to the group of units of $\mathcal{O}_2$ (see \cite{Rav}[Lemma A2.2.20]).  Recall from Section~\ref{Sec:FGL} the map
\[ \widetilde{\theta_{\Gamma_n}}: \widetilde{VT} \to Map^c( \St, (E_2)_0),\]
where $\Gamma_n$ is the height $n$ Honda formal group law. To avoid cumbersome notation, let us continue (from Section~\ref{Sec:FGL}) to  denote $\widetilde{\theta_{\Gamma_n}}(\widetilde{t_n})$ simply by $\widetilde{t_n}$. From \eqref{timodmax}, we learn that if \[ \gamma^{-1} = {\sum\limits_{n \geq 0}}^{\Gamma_n}\overline{a}_nx^{2^n} \in \St,\] then  
\begin{equation} \label{eqn:tn}
 \widetilde{t_n}(\gamma) \equiv a_n  \hspace{-8pt}\mod (2,u_1).
\end{equation} 
Also keep in mind that the Teichm\"uller lifts $a_n$ satisfy the equation $a_n^4 = a_n$. Therefore we have 
\begin{equation} \label{powercong}
\widetilde{t_n}(-)^4 \equiv \widetilde{t_n}(-) \hspace{-8pt}\mod (2,u_1) 
\end{equation}

We can also write every element of $\mathcal{O}_2$ as $a + b T$ where $a,b \in  W(\Ff)$. Using the isomorphism $\St \iso \mathcal{O}_2^{\times}$, one defines a determinant map (see \cite[\S 2.3]{Bea1})
\[\det: \St \to  \Z_2^{\times} \]
which sends $a + bT \mapsto aa^{\sigma} - 2bb^{\sigma}$.
Composing this with the quotient map $\Z_2^{\times}\to\Z_2^{\times}/\{\pm1\}\iso\Z_2$ gives the \emph{norm map}
\begin{equation} 
\label{norm} N:\St \overset{\mathit{det}}\longrightarrow \Z_2^{\times} \twoheadrightarrow \Z_2.
\end{equation}
The kernel of the norm map $N: \St \to \Z_2$ is called the \emph{norm one} subgroup and is denoted by $\Sto$.  In \cite{Bea1}[\S 2.3], Beaudry described two elements $\alpha,\pi \in \St$ such that $\det(\alpha) = -1$ and $\det(\pi) = 3$, two elements which generate $\mathbb{Z}_2^{\times}$ topologically. In particular, $\pi$ defines an isomorphism $\St \iso \Sto \rtimes \mathbb{Z}_2$. 


As we will see in this section, the most crucial part of the action of $\St$ on $\E Z$ is the action of its finite subgroups, which we describe here, following \cite{Bea1} and \cite{Buj}.

\begin{prop}
Every maximal finite nonabelian subgroup of $\St$ is conjugate to a group\[G_{24}=Q_8\rtimes C_3,\]where $Q_8$ is the quaternion group\[Q_8=\langle\ii,\jj:\ii^4=\ind,\ii^2=\jj^2,\ii^3\jj=\jj\ii\rangle\]and $C_3$ acts by permuting $\ii,\jj,$ and $\kk=\ii\jj$.
\end{prop}

\begin{notn}
We denote the identity element of $Q_8$ by $\ind$. The order $2$ element of $Q_8$ is often denoted by $-\ind$, however, to avoid confusing it with an element of a ring, we will denote it by $\nind\in Q_8$. Similarly, we denote $\hi=\nind\ii,\hj=\nind\jj,\hk=\nind\kk$. The center of $Q_8$, which is an order 2 group generated by $\nind$, will be denoted $C_{\nind}$.
\end{notn}
The maximal finite subgroup $G_{24}$ is unique up to conjugation as a subgroup of $\St$, while as a subgroup of $\Sto$, there are two conjugacy classes, $G_{24}$ and $G_{24}'=\pi G_{24}\pi^{-1}$. The group $\Sto$ also has a cyclic subgroup\[C_6=C_{\nind}\times C_3\]generated by $\nind=\ii^2$ and $\omega$.

The identification of $\St$ with $\mathcal{O}_2^{\times}$ endows $\St$ with a decreasing filtration
\begin{equation}\label{St:res}F_{0/2}\St=\St, F_{n/2}\St = \lbrace \gamma \in \St: \gamma \equiv {1}\pmod{T^n}) \rbrace.\end{equation}
One should note that  
\[ 
F_{n/2} \St / F_{(n+1)/2} \St \iso\left\lbrace \begin{array}{ccc}
C_3 & n=0 \\
\mathbb{F}_4 & n\geq 0. 
 \end{array} \right. \]
Notice from \eqref{eqn:tn}  that $t_n$  co-acts trivially on $\BP Z$ for $n>2$. Therefore, following \eqref{eqn:tn} we conclude that $F_{n/2}\St$ for $n>2$ acts trivially on $\E Z$. We list the generators of the various filtration quotients of $\St$ in the following table:
\[\begin{tabular}{|c|c|} \hline
\mbox{associated graded}&\mbox{generators} \\ \hline 
\rule{0pt}{3ex} $F_{0/2}\St/F_{1/2}\St\iso C_3$&$\omega$\\ 
\rule{0pt}{3ex} $F_{1/2}\St/F_{2/2}\St\iso Q_8/C_{\nind}\iso C_2\times C_2$&$\ii,\jj$\\ 
\rule{0pt}{3ex} $F_{2/2}\St/F_{3/2}\St\iso C_{\nind}\times C_{\alpha}$&$\nind,\alpha$\\ 
\hline
\end{tabular}\]
The subgroup 
 \[K=\overline{\langle\alpha,F_{3/2}\St\rangle},\]
 is known as the Poincar\'e duality subgroup of $\St$, and it is known that 
\[\St\iso K\rtimes G_{24}\]
The subgroup $\Sto$ inherits the filtration of $\St$ via $F_{n/2} \Sto := \Sto \cap F_{n/2}\St$. In particular, we have 

\[ F_{n/2} \Sto / F_{(n+1)/2} \Sto \iso\left\lbrace \begin{array}{ccc}
C_3 & n=0 \\
\mathbb{F}_4 & \text{$n\geq 0$ and $n$ odd, or $n=2$ }\\
\mathbb{F}_2 & \text{otherwise.}
 \end{array} \right.
\]
There is a corresponding  Poincar\'e duality subgroup $K^1$ which satisfies
\[\Sto\iso K^1\rtimes G_{24}.\]

Recall from \eqref{eqn:BP2Z} that 
\[ \E Z \iso \E \otimes_{\BP}\BP Z \iso \E /(2,u_1)\langle x_0, x_2, x_4, x_6, y_6, y_8,y_{10}, y_{12} \rangle. \]
For our purposes it is convenient to have all the generators in degree $0$, so we define 
\[ \x_i = u^{\frac{i}{2}}x_i , \y_i = u^{\frac{i}{2}}y_i\]
in order to have
\[\E Z  \iso  \E /(2,u_1)\langle \x_0, \x_2, \x_4, \x_6, \y_6, \y_8, \y_{10}, \y_{12}\rangle.\]

By~\eqref{Zaction}, the action of $\St$ can be expressed in terms of maps $\widetilde{t_i}:\St\to E(\Ff,\Gamma_2)$. For instance, we have\[\psi_Z^{BP}(x_2)=1|x_2+t_1|x_0=1|x_2+u^{-1}\widetilde{t_0}^{-2}\widetilde{t_1}|x_0,\]so for $\gamma\in\St$, we have
\[\gamma(\x_2)= \gamma (u x_2)=\widetilde{t_0}(\gamma)u(x_2+u^{-1}\widetilde{t_0}(\gamma)^{-2}\widetilde{t_1}(\gamma)x_0)=\widetilde{t_0}(\gamma)\x_2+\widetilde{t_0}(\gamma)^{-1}\widetilde{t_1}(\gamma)\x_0.\]
By Lemma \ref{lem:independant}, if we suppress `$\gamma$', the action of $\St$ can be described as in Table~\ref{eqn:moravaaction} (thanks to \eqref{powercong}, we can adopt the simplifications $\widetilde{t_i}^4=\widetilde{t_i}$ for $i=1,2$ and $\widetilde{t_0}^3=1$).%
\begin{table}[!h] 
$\begin{array}{|c|c|}
\hline \rule{0pt}{3ex} 
u&\widetilde{t_0}  u\\
\hline \rule{0pt}{3ex} 
\x_0&\x_0\\
\hline \rule{0pt}{3ex} 
\x_2&\widetilde{t_0}^{2}\widetilde{t_1}  \x_0 + \widetilde{t_0}  \x_2\\ 
\hline \rule{0pt}{3ex} 
\x_4&\widetilde{t_0}^{1}\widetilde{t_1}^2\x_0 + \widetilde{t_0}^2\x_4\\
\hline \rule{0pt}{3ex}
\x_6&\widetilde{t_1}^3\x_0 + \widetilde{t_0}^{2}\widetilde{t_1}^2\x_2 + \widetilde{t_0}\widetilde{t_1}\x_4 +\x_6\\
\hline \rule{0pt}{3ex}
\y_6&(\widetilde{t_1}^3+\widetilde{t_0}^{2}\widetilde{t_2})\x_0 + \widetilde{t_0}^{2}\widetilde{t_1}^2\x_2 + \y_6\\
\hline \rule{0pt}{3ex}
\y_8&(a\widetilde{t_0}^{2}\widetilde{t_1} + \widetilde{t_0}^{}\widetilde{t_1}\widetilde{t_2})\x_0 + \widetilde{t_2}\x_2 + \widetilde{t_1}^2\x_4 + \widetilde{t_0}^2\widetilde{t_1}\y_6 + \widetilde{t_0}\y_8\\
\hline \rule{0pt}{3ex}
\y_{10}&(\widetilde{t_0}\widetilde{t_1}^2+\widetilde{t_1}^2\widetilde{t_2})\x_0 + \widetilde{t_1}\x_2 + (\widetilde{t_0}^{2}\widetilde{t_1}^3+\widetilde{t_0}\widetilde{t_2})\x_4 + \\\rule{0pt}{2ex}
&+ \widetilde{t_0}\widetilde{t_1}^2\x_6 + \widetilde{t_0}\widetilde{t_1}\y_6 + \widetilde{t_0}^2\y_{10}\\ 
\hline \rule{0pt}{3ex}
\y_{12}&((b+1)\widetilde{t_1}^3+\widetilde{t_0}^{2}\widetilde{t_1}^3\widetilde{t_2}+(a+b)\widetilde{t_0}\widetilde{t_2}^2)\x_0 + \widetilde{t_0}\widetilde{t_1}^2\widetilde{t_2}\x_2 + \\ \rule{0pt}{2ex}
&+ (b\widetilde{t_0}^{}\widetilde{t_1}+\widetilde{t_1}\widetilde{t_2})\x_4 + \widetilde{t_0}^2\widetilde{t_2}\x_6 + \widetilde{t_1}^3\y_6 + \widetilde{t_0}^2\widetilde{t_1}^2\y_8 + \widetilde{t_0}\widetilde{t_1}\y_{10} + \y_{12}\\ 
\hline
\end{array}$
\caption{The action of $\St$ on $\E Z$}
\label{eqn:moravaaction}
\end{table}

 Let $\overline{M}_* = \E \otimes_{\BP} M_*$, where $M_*$ is the $\BPBP$-comodule introduced in Remark~\ref{rem:M}. Define 
\[ \overline{g}_i  = u^{\frac{i}{2}} g_i\]
to form a set of generators $\{\overline{g}_0,\overline{g}_2,\overline{g}_4,\overline{g}_6\}$ of 
$\overline{M}_0$. 
A consequence of Lemma~\ref{BPexactZ} and Table~\ref{eqn:moravaaction} is: 
\begin{lem} \label{exactM}There is an exact sequence 
\begin{equation} \label{exactmoravamod}
 0 \to  \overline{M}_0 \overset{\overline{\iota}}\to \Ez Z \overset{\overline{\tau}}\to \overline{M}_0 \to 0
\end{equation}
of $Q_8$-modules, where $\overline{\iota}(\overline{g}_i) = \x_i$ and $\overline{\tau}(\y_i) = \overline{g}_{i-6}$. 
\end{lem} 
We will use the exact sequence of \eqref{exactmoravamod} (compare with \eqref{exact}) to understand the action of $Q_8$ on $\Ez Z$. For this purpose we need the data of $\widetilde{t_i}(\gamma)$ modulo $(2,u_1)$ for $\gamma \in Q_8$. By definition of the Honda formal group law $\Gamma_2$, we have $[2]_{\Gamma_2}(x)=x^4$ and it follows that
\[[-1]_{\Gamma_2}(x)=\sum\nolimits_{n\geq0}^{\Gamma_2}x^{4^n}.\]
Indeed, one has
\[x+_{\Gamma_2}\sum\nolimits_{n\geq0}^{\Gamma_2}x^{4^n}=x+_{\Gamma_2}x+_{\Gamma_2}\sum\nolimits_{n\geq1}^{\Gamma_2}x^{4^n}=x^4+_{\Gamma_2}\sum\nolimits_{n\geq1}^{\Gamma_2}x^{4^n}=\cdots=0.\]
Following \eqref{timodmax} and using the fact that $[-1]_{\Gamma_2}(x)$ is its own inverse, we have 
\[\widetilde{t_n}(\nind) = 
\left\{\begin{array}{rl}1&n\mbox{ even}\\0&n\mbox{ odd}\end{array}\right.
\]
modulo $(2,u_1)$. Further, $\ii$ and $\jj$ can be chosen so that modulo $(2,u_1)$, one has
\begin{itemize}
\item $\widetilde{t_0}(\gamma) = 1$ for all $\gamma \in Q_8$
\item $\widetilde{t_1}(\ii) = \widetilde{t_1}(\hi) =1, \widetilde{t_1}(\jj) = \widetilde{t_1}(\hj) = \omega, \text{ and } \widetilde{t_1}(\kk) = \widetilde{t_1}(\hk) =\omega^2$. 
\end{itemize}
Let $\Ff[Q_8/C_{\nind}]$ and $\Ff[Q_8]$ denote the group rings over $\Ff$.
\begin{lem} \label{repM} There is an isomorphism $\overline{M}_0 \iso \Ff[Q_8/C_{\nind}]$ of $\Ff[Q_8]$-modules.
\end{lem}
\begin{proof} Since $\overline{\iota}(\overline{g}_i) = \x_i$, one can read the action of $Q_8$ off Table~\ref{eqn:moravaaction}. With respect to the ordered basis $ \mathcal{B} =  \set{\overline{g}_0 , \overline{g}_2,\overline{g}_4, \overline{g}_6}$ of $\overline{M}_0$, we have 
\[(\ii)_{\mathcal{B}}= (\hi)_{\mathcal{B}} = \left[\begin{array}{cccc}1&1&1&1\\0&1&0&1\\0&0&1&1\\0&0&0&1\end{array}\right], 
(\jj)_{\mathcal{B}}=(\hj)_{\mathcal{B}} = \left[\begin{array}{cccc}1&\omega&\omega^2&1\\0&1&0&\omega^2\\0&0&1&\omega\\0&0&0&1\end{array}\right], \]
 \[ (\kk)_{\mathcal{B}}= (\hk)_{\mathcal{B}} =\left[\begin{array}{cccc}1&\omega^2&\omega&1\\0&1&0&\omega\\0&0&1&\omega^2\\0&0&0&1\end{array}\right].\]
Consider the basis $\mathcal{C}=\set{v_0, v_1, v_2, v_3}$ of $\overline{M}_0$ where 
\[Q=\mtrx{1&0&0&1\\0&1&1&1\\0&\omega&\omega^2&1\\0&0&0&1}\]
is the change of basis matrix from $\mathcal{B} \to \mathcal{C}$.
 It can be readily checked that 
\[(\ii)_{\mathcal{C}}= Q^{-1} (\ii)_{\mathcal{B}} Q =  \mtrx{1&1&0&0\\0&1&0&0\\0&0&1&1\\0&0&0&1},
(\jj)_{\mathcal{C}}=  Q^{-1} (\jj)_{\mathcal{B}} Q = \mtrx{1&0&1&0\\0&1&0&1\\0&0&1&0\\0&0&0&1}, \]
\[ (\kk)_{\mathcal{C}}=  Q^{-1} (\kk)_{\mathcal{B}} Q = \mtrx{1&1&1&1\\0&1&0&1\\0&0&1&1\\0&0&0&1}.\] 
The last three matrices are equal to those representing the same transformations in the basis $\mathcal{B}_4$ of $\Ff[Q_8/C_{\nind}]$ given in \eqref{appendix:B4}, and thus we conclude that $\overline{M}_0$ and $\Ff[Q_8/C_{\nind}]$ are isomorphic as $\Ff[Q_8]$-modules.
\end{proof}
\begin{thm} \label{crucial}
There is an isomorphism $\Ez Z\iso \Ff[Q_8]$ of $\Ff[Q_8]$-modules.
\end{thm}
\begin{proof}
Let  $\mathcal{B},\mathcal{C}$ and $Q$ be as in the proof of Lemma \ref{repM}. Let \[\mathcal{B}_Z=\{\x_0, \x_2,\x_4,\x_6,\y_6,\y_8,\y_{10},\y_{12}\}\]be the usual ordered basis of $\Ez Z$ and let $\mathcal{C}_Z= \set{c_0, c_1, c_2, c_3, c_0', c_1', c_2', c_3'}$ be another  basis of $\Ez Z$ where  
$\tilde{Q} = \mtrx{ 
Q & * \\ 
0 & Q
}$
is a change of basis matrix from $\mathcal{B}_Z \to \mathcal{C}_Z$. 

By Lemma \ref{exactM}, in the exact sequence \eqref{exactmoravamod}, we have $\overline{\iota}(v_i) = c_i$ and $\overline{\tau}(c_i') = v_i$. From Table~\ref{eqn:moravaaction}, we know that
\[(\nind)_{\mathcal{B}_Z}=\mtrx{I&M\\0&I},\]where\[M=\mtrx{1&0&0&a+b\\0&1&0&0\\0&0&1&0\\0&0&0&1}.\]
Then\[(\nind)_{\mathcal{C}_Z}=\tilde{Q}^{-1}(\nind)_{\mathcal{B}_8}\tilde{Q}=\mtrx{I&Q^{-1}MQ\\0&I}=\mtrx{I&M\\0&I}.\]

The remainder of the proof that $\Ez Z\iso V_8(\Ff)$ will continue in Lemma~\ref{appendixmain}.
\end{proof}
Because $Q_8$ acts trivially on $u$ we get the following corollary.
\begin{cor} \label{cor:crucial} There is an isomorphism 
\[\E Z\iso\Ff[u^{\pm1}][Q_8]\]of graded $Q_8$-modules.
\end{cor}

\begin{lem} \label{pi} The element $\alpha^{-1}\pi \in \St$ acts trivially on the generators of $\E Z$. 
\end{lem}
\begin{proof} By definition of $\alpha$ and $\pi$ \cite{Bea1}[\S 2.3], $\pi \equiv \alpha \bmod{F_{3/2} \St}$. Therefore, $\alpha^{-1} \pi \in  F_{3/2} \St$. The result follows from the fact that $t_n$ coacts trivially on $\BP Z$  (i.e. the $\BPBP$-comodule structure of $\BP Z$ does not contain any terms involving $t_n$ for $n\geq 3$.)
\end{proof}
Further note that 
\[ \text{$\widetilde{t_0}(\alpha) \equiv 1$, $\widetilde{t_1}(\alpha) \equiv 0$ and $\widetilde{t_2}(\alpha) \equiv \omega\bmod{(2,u_1)}.$} \] 
A direct inspection of Table~\ref{eqn:moravaaction} shows that $\alpha$ acts nontrivially on $\E Z$, in fact we have:     
\begin{cor}\label{fixed:alpha-1}
The fixed point modules $\E Z^{C_{\nind}}$ and $\E Z^{C_{\alpha}}$ both equal $\Ff[u^{\pm1}]\{\x_0,\x_2,\x_4,\x_6\}$.
\end{cor}
\begin{cor} \label{trivialK}
The subgroup $F_{2/2}\St$ acts trivially on $\E Z^{C_{\nind}}$.
\end{cor}
\begin{proof}
We know that $F_{3/2}\St$ acts trivially as  $t_n$ coacts trivially on $\BP Z$ for $n\geq3$. Furthermore, $F_{2/2}\St/F_{3/2}\St$ is generated by $\nind$ and $\alpha$, both of which act trivially on $\E Z^{C_{\nind}}$.
\end{proof}

\section{The duality resolution spectral sequence for $Z$} \label{Sec:Duality}
Now that we have complete knowledge of the action of $\Sto$ on $\E Z$, we are all set to calculate the group cohomology $H^*(\Sto; \E X)$, which is the key step to finding the $E_2$ page of the descent spectral sequence \eqref{eqn:descent}. 
We will use the duality resolution spectral sequence, a convenient tool to calculate the $E_2$-page. The duality resolution spectral sequence comes from the duality resolution, which is a finite $\Z_2[[\Sto]]$-module resolution of $\Z_2$. First we fix some notations. 
\begin{notn}
Throughout this section, we will let 
\begin{itemize}
\item $S_2:=F_{1/2} \St$,
\item $S_2^1:=F_{1/2}\Sto$, and, 
\item $IS_2^1$ be the augmentation ideal of the group ring $\Z_2[[S_2^1]]$.
\end{itemize}
Note that every element in $IS_2^1$ can be written as an infinite sum of elements of the form $a_g(\ind-g)$, where $\ind$ denotes the neutral element of $\St$, $g\in S_2^1$ and $a_g\in\Z_2[[S_2^1]]$.
\end{notn}
\begin{thm}[Goerss-Henn-Mahowald-Rezk, Beaudry~\cite{Bea2}]\label{dualityresolution} Let $\Z_2$ be the trivial $\Z_2[[ \Sto]]$-module. There is an exact sequence of complete left $\Z_2[[\Sto]]$-modules. 
\[ 
\xymatrix{
0 \ar[r]& \mathscr{C}_3 \ar[r]^-{\partial_3} & \mathscr{C}_2 \ar[r]^-{\partial_2} & \mathscr{C}_1 \ar[r]^-{\partial_1} & \mathscr{C}_0 \ar[r]^-{\epsilon} & \Z_2\ar[r] & 0,
}\]
where $\mathscr{C}_0 \iso \Z_2[[\Sto/G_{24}]]$, $\mathscr{C}_3 \iso \Z_2[[\Sto/G_{24}']]$ and $\mathscr{C}_1 \iso \mathscr{C}_2 \iso \Z_2[[\Sto/C_6]]$. Let $e$ be the unit in $\Z_2[[\Sto ]]$ and $e_p$ be the resulting generator in $\mathscr{C}_p$. The map $\partial_p$ can be chosen to satisfy: 
\begin{itemize}
\item $\partial_1(e_1) = (e-\alpha) \cdot e_0,$ 
\item $\partial_2(e_2) \equiv (e + \alpha) \cdot e_1 \mod (2, (IS_2^1)^2)$,
\item $ \partial_3(e_3) = \pi(e+i+j+k)(e- \alpha^{-1}) \pi^{-1} \cdot e_2$.
\end{itemize}
Let $F_0 = G_{24}$, $F_1 = F_2 = C_6$ and $F_3 = G_{24}'$. For a profinite $\Z_2[[ \Sto]]$-module $M$, there is a first quadrant spectral sequence 
\[ E_1^{p,q} = Ext_{\Z_2[[\Sto]]}(\mathscr{C_p}, M) \iso H^q(F_p; M) \Rightarrow H^{p+q}(\Sto; M) \]
with differentials $d_r: E_1^{p,q} \to E_1^{p+r, q-r +1}$. 
\end{thm}

Since the map $BP_* \to \E$ sends $v_2 \mapsto u^{-3}$, we will denote $u^{-3}$ by $v_2$. Let us now recall Shapiro's lemma, an important result in group cohomology, which will be used throughout the rest of this section.
\begin{lem}[Shapiro's Lemma] \label{shapiroshapiro}
Let $G$ be a finite group, $H\subset G$ a normal subgroup and let $M$ be a $\mathbb{Z}[H]$-module. Then for every $n$, we have
\[H^n(G;\mathit{CoInd}_H^G (M)])=H^n(H;M), \]
where $\mathit{CoInd}_H^G (M) = Hom_{\Z[H]}(Z[G], M)$. 
\end{lem}
\begin{rem} If $H \subset G$ is a subgroup of finite index, then 
\[ \mathit{CoInd}_H^G(M) \iso \mathit{Ind}_H^G(M) := \Z[G] \otimes_{\Z[H]} M .\]
In all of our applications of Lemma~\ref{shapiroshapiro}, $G$ will be a finite group, hence one need not distinguish between $\mathit{CoInd}_H^G(M)$ and $\mathit{Ind}_H^G(M)$ for any normal subgroup $H$ of $G$. 
\end{rem}
Theorem \ref{cor:crucial} along with Shapiro's lemma \ref{shapiroshapiro} implies that
\[ 
H^p(Q_8;\E Z) \iso \left\lbrace 
\begin{array}{ccccccc}
\Ff[u^{\pm1}] & p = 0 \\
0 & \text{otherwise.}
\end{array}
\right.
\] Furthermore, since $\E Z$ has characteristic $2$ and $Q_8$ is the $2$-Sylow of $G_{24}$, we have 
\[H^p(G_{24};\E Z)\iso H^p(Q_8;\E Z)^{C_3} \iso \Ff[u^{\pm1}]^{C_3}.\]
The generator $\omega$ of $C_3$ acts non-trivially on $u$ (see Table~\ref{eqn:moravaaction}), but fixes $u^3$ so that  
\begin{equation} \label{G24Z}
H^p(G_{24};\E Z)=\left\{\begin{array}{ll}\Ff[u^{\pm3}]&p=0\\0&p\neq0\end{array}\right.
\end{equation}
\begin{rem}Note that the equivalence 
\begin{equation} \label{tmfZ}
\mathit{tmf} \sma Z \simeq k(2)
\end{equation}
in the $K(2)$-local category of spectra can be seen as a consequence of \eqref{G24Z}. For the maximal finite subgroup $G_{48} \subset \Gt$, it is well known that $(E_2)^{hG_{48}} \iso L_{K(2)} \mathit{tmf}$. The $E_2$-page of the homotopy fixed point spectral sequence for computing the homotopy groups of $(E_2)^{hG_{48}} \sma Z$ is precisely  $H^p(G_{48};\E Z)=H^p(G_{24};\E Z)^{Gal(\Ff/\Ft)}$, which, by \eqref{G24Z}, is isomorphic to $\Ft[u^{\pm 3}]$ and hence collapses. Thus we get 
\[ \pi_*(L_{K(2)}(\mathit{tmf} \sma Z))\iso K(2)_* .\]
\end{rem}
\begin{lem} \label{G24'} Let $G_{24}' = \pi G_{24} \pi^{-1}$ in $\Sto$. Then we have 
\[ H^{q}(G_{24}'; E_*Z) \iso H^{q}(G_{24}; E_*Z) \iso \Ff[v_2^{\pm1}].\]
\end{lem}
\begin{proof} Notice that $G_{24}' \cap G_{24}  \supset C_{\nind}$. Also keep in mind that $\pi \equiv \alpha \mod F_{3/2}\Sto$. Therefore, by Corollary~\ref{fixed:alpha-1}, $\pi$ acts trivially on $(\Ez Z)^{C_{\nind}}$. Let $Q_8' = \pi Q_8 \pi^{-1} \subset G_{24}'$,  $\ii' = \pi \ii \pi^{-1}$ and $\jj' = \pi \jj \pi^{-1}$. Thus we have 
\[ G_{24}' \iso Q_8' \rtimes C_3\]
Note that 
\[ \ii' \equiv \ii \hspace{-4pt}\mod F_{2/2}\Sto \text{ and } \jj' \equiv \jj  \hspace{-4pt} \mod F_{2/2}\Sto \]
because $\pi \in F_{2/2} \Sto$. Therefore, the actions of $\ii'$ and $\jj'$ on $(\Ez Z)^{C_{\nind}}$ are exactly the same as that of  $\ii$ and $\jj$ respectively. It follows that 
\[ (\Ez Z)^{C_{\nind}} \iso \Ff[Q_8'/C_{\nind}] \]
as a $\Ff[Q_8']$-module. Applying the arguments of Theorem~\ref{crucial}  to this case, one sees 
\[ \Ez Z \iso \Ff[Q_8']\]
as an $\Ff[Q_8']$-module, and the result follows.
\end{proof}

Now we shall focus on computing $H^{q}(C_6; \E Z)$. Take $C_{\nind}$ to be the center of $Q_8$ and consider $C_6=C_{\nind}\times\Ff^{\times}$. While $C_{\nind}$ fixes all the $\x_i$ in addition to fixing $u$, the group $\Ff^{\times}$ does not fix the $\x_i$; however it does fix the $x_i$. This observation will be crucial for the computation that follows.  Because $C_{\nind}$ is the 2-Sylow subgroup of $C_6$, we have
\[H^q(C_6;\E Z)=H^q(C_{\nind};\Ff[u^{\pm1}][Q_8])^{C_3}.\]
Because $\E Z \iso \Ff[u^{\pm1}][Q_8]$ is a free $\Ff[C_{\nind}]$-module we have 
\begin{eqnarray*}
H^q(C_{\nind};\Ff[u^{\pm1}][Q_8])^{C_3} &\iso& ((\E Z)^{C_{\nind}})^{C_3} \\
&\iso&  (\Ff[u^{\pm 1}][ \overline{x}_0, \overline{x}_1, \overline{x}_2, \overline{x}_3])^{C_3}\\
&\iso&  \Ff[u^{\pm 3}][ x_0,x_1,x_2,x_3]
\end{eqnarray*}
concentrated at $q=0$. Essentially deriving from Table~\ref{eqn:moravaaction}, we list the actions of $\ii$, $\jj$, $\ii\jj$ on the generators $x_0,x_2,x_4,x_6$, which will come in handy later on. 
\begin{equation} \label{ijkaction}
\begin{array}{|c|c|c|c|}
\hline
x& \ii \cdot x& \jj \cdot x & \ii\jj \cdot x\\
\hline
x_0& x_0 & x_0 & x_0\\
\hline
x_2&  u^{-1}x_0 + x_2& \omega u^{-1} x_0 + x_2 & \omega^{2}u^{-1}x_0 + x_2   \\
\hline
x_4& u^{-2}x_0 + x_4 & \omega^2 u^{-2}x_0 + x_4 & \omega u^{-2}x_0 + x_4  \\
\hline
x_6&  u^{-3}x_0 + u^{-2}x_2  & u^{-3}x_0 + \omega^2 u^{-2}x_2 & u^{-3}x_0 + \omega u^{-2}x_2\\
 & + u^{-1}x_4 + x_6    & + \omega u^{-1}x_4 + x_6 & + \omega^2 u^{-1}x_4 + x_6   \\
\hline
\end{array}
\end{equation}

To summarize and as well as to establish notations, we rewrite the $E_1$-page of the duality resolution spectral sequence for $Z$ as 
\[E_1^{p,q}=\left\{\begin{array}{rl}\Ff[v_2^{\pm1}]\langle x_{0,0} \rangle &p=0, q=0\\
\Ff[v_2^{\pm1}]\langle x_{1,0},x_{1,2},x_{1,4}, x_{1,6}\rangle&p=1, q=0\\
\Ff[v_2^{\pm1}]\langle x_{2,0}, x_{2,2},x_{2,4}, x_{2,6}\rangle&p=2, q = 0\\
\Ff[v_2^{\pm1}]\langle x_{3,0} \rangle &p=3, q=0 \\
0 & \text{otherwise}\end{array} \right.\]
where the internal grading of $x_{i,j}$ is $j$. To compute the differentials in this spectral sequence, we need the following result.
\begin{thm} \label{S21Z} For every $Z \in \ZZ$, $H^*(S_2^1;\E Z)$ is isomorphic to  
 \[H^{*}(K^1; \Ft) \otimes \Ff[u^{\pm 1}] \]
as an $\Ff[u^{\pm 1}]$-module.
\end{thm}
 \begin{proof}
We begin with the calculation of  $H^*(F_{2/2}\Sto; \E Z)$. Since 
$$F_{2/2}\Sto  \iso  K^1 \times C_{\nind},$$ 
we have a Lyndon-Hochschild-Serre spectral sequence 
\begin{equation} \label{E2LHS}
 E_2^{p,q} = H^p(K^1 ;H^{q}(C_{\nind}; \E Z)) \To H^{p+q}(F_{2/2}\Sto; \E Z).
\end{equation}
Since $H^{q}(C_{\nind}; \E Z) \iso (\E Z)^{C_{\nind}}$ (concentrated at $q=0$) and $K^1$ acts trivially on  $(\E Z)^{C_{\nind}}$ by Corollary~\ref{trivialK}, the spectral sequence \eqref{E2LHS} collapses and we have 
\[ H^*(F_{2/2}\Sto; \E Z) \iso H^*(K^1; \Ft) \otimes \E Z^{C_{\nind}}. \]
Note that $$\E Z^{C_{\nind}} \iso \Ff[u^{\pm 1}]\langle \x_0, \x_2, \x_4, \x_6 \rangle \iso \Ff[u^{\pm 1}][Q_8/C_{\nind}]. $$ Now we run yet another Lyndon-Hochschild-Serre spectral sequence
\begin{equation} \label{LHS2}
E_2^{p,q} = H^{p}(Q_8/C_{\nind}; H^q(F_{2/2}\Sto; \E Z)) \Rightarrow H^{p+q}(S_2^1; \E Z)
\end{equation}
to compute $H^{\ast}(S_2^1; \E Z)$. Notice that 
\begin{eqnarray*}
 E_2^{p,q} &=& H^p(Q_8/C_{\nind}; H^q(F_{2/2}\Sto; \E Z)) \\ 
 &=& H^p(Q_8/C_{\nind}; H^q(K^1; \Ft) \otimes \E Z^{C_{\nind}})  \\
&=&  H^p(Q_8/C_{\nind}; H^q(K^1; \Ft) \otimes \Ff[u^{\pm 1}][Q_8/C_{\nind}]) \\
&=& \left\lbrace \begin{array}{ccc}
 H^q(K^1; \Ft) \otimes \Ff[u^{\pm 1}] &\text{when $p=0$} \\
0 & \text{when $p \neq 0$} 
\end{array} \right.
\end{eqnarray*}
by Shapiro's Lemma~\ref{shapiroshapiro}. Thus the spectral sequence \eqref{LHS2} collapses at the $E_2$-page and we get  
\[ H^{*}(S_2^1; \E Z) = H^{*}(K^1; \Ft) \otimes \Ff[u^{\pm 1}].\]
 \end{proof}
From the above theorem and the following unpublished result of Goerss and Henn (see~\cite{Bea1}[Theorem $2.5.13$]), 
\begin{equation} \label{henn}
H^*(K^1;\Ft) \iso \Ft[y_0,y_1,y_2]/(y_0^2,y_1^2+y_0y_1,y_2^2+y_0y_2).
\end{equation}
we get a complete description of $H^*(S_2^1;\E Z)$.

Our next goal is to make use of the formulas in Theorem~\ref{dualityresolution} to  calculate the $d_1$-differentials of the duality resolution spectral sequence for $Z$. Moving forward, there are two things that are handy to keep in mind:
\begin{itemize}
\item $Z$ admits a $v_2^1$-self-map~\cite{BE}, therefore differentials in the duality resolution spectral sequence for $Z$ will be $v_2^1$-linear. 
\item The $d_1$-differentials preserve the internal grading.
\end{itemize}
\begin{lem}\label{GHMR:d1}
The differentials $d_1:E_1^{0,0}\to E_1^{1,0}$ and $ d_1:E_1^{2,0}\to E_1^{3,0}$ are zero, while the differential $d_1:E_1^{1,0}\to E_1^{2,0}$ is the $v_2$-linear map given by
\begin{eqnarray*}
 \Ff[v_2^{\pm1}]\langle x_{1,0},x_{1,2},x_{1,4}, x_{1,6}\rangle&\longrightarrow& \Ff[v_2^{\pm1}]\langle x_{2,0},x_{2,2},x_{2,4}, x_{2,6}\rangle\\
x_{1,0},x_{1,2},x_{1,4}&\mapsto&0\\
x_{1,6}&\mapsto& \lambda v_2x_{2,0},
\end{eqnarray*}
where $\lambda \in \Ff^{\times}$.  The duality resolution spectral sequence for $Z$ collapses at the $E_2$-page. 
\end{lem}
\begin{proof}
It follows from Theorem~\ref{dualityresolution} that the differential $d_1:E_1^{0,0}\to E_1^{1,0}$ is given by
\[d_1(x)=(\ind-\alpha)\cdot x,\]
which is zero, because $\alpha$ fixes $x_{0,0}$, (follows from Table~\ref{eqn:moravaaction}, also see Corollary~\ref{fixed:alpha-1}). Likewise, the differential $d_1:E_1^{2,0}\to E_1^{3,0}$ is given by
\[d_1(x)=\pi(\ind+\ii+\jj+\kk)(\ind-\alpha^{-1})\pi^{-1} \cdot x,\]
which is zero because by Lemma~\ref{fixed:alpha-1} and the fact that $\alpha\equiv\pi\mod{F_{3/2}\Sto}$, both $\alpha$ and $\pi$ fix all the $x_{1,i}$. The differential $d_1:E_1^{1,0}\to E_1^{2,0}$ is given by
\[d_1(x)=\Theta \cdot x=(\ind+\alpha+\mathcal{E}) \cdot x,\]
where $\mathcal{E}\in(IS_2^1)^2$. Because $\alpha$ fixes all the $x_{1,j}$, this simplifies to\[d_1(x)=\mathcal{E} \cdot x.\]The element $\mathcal{E}$ is a possibly infinite sum of the form
\[\mathcal{E}=\sum a_{g,h}(\ind-g)(\ind-h)\]for $a_{g,h}\in\Z_2[[S_2^1]]$ and $g,h\in S_2^1$. In particular, thanks to Table~\ref{eqn:moravaaction} and \eqref{product-tk}, we know that
\begin{eqnarray*}
(\ind-g)(\ind-h)\cdot x_{0}&=&0\\
(\ind-g)(\ind-h)\cdot x_{2}&=&0\\
(\ind-g)(\ind-h)\cdot x_{4}&=&0\\
(\ind-g)(\ind-h)\cdot x_{6}&=&(\widetilde{t_1}(h)\widetilde{t_1}(g)^2 + \widetilde{t_1}(h)^2\widetilde{t_1}(g))x_0,
\end{eqnarray*}
and it follows that\[d_1(x_{1,0})=d_1(x_{1,2})=d_1(x_{1,4})=0,\]while $d_1(x_{1,6})$ is a multiple of $x_{2,0}$. We know that $d_1$ is not identically zero, because 
\[H^1(\Sto; \E Z) \iso H^1(S_2^1; \E Z)^{C_3}\]
has rank at most $3$. Since differentials preserve internal grading 
\[ \mathcal{E} \cdot x_{1,6} = \lambda v_2x_{2,0},\]
where $\lambda \in \Ff^{\times}$, is forced.  Since $E_1^{p,q} = 0 $ for $q \neq 0$, the duality resolution spectral sequence for $Z$ collapses at the $E_2$-page.  
\end{proof}

\begin{rem} \label{fixC3K}
Since $H^{p}(\Sto; \E Z) \iso H^{p}(S_2^1; \E Z)^{C_3}  \iso (H^{p}(K^1; \Ft) \otimes \Ff[u^{\pm 1}])^{C_3},$
one requires an understanding of the action of $C_3$ on  $H^{p}(K^1; \Ft)$. This action is given by 
\begin{eqnarray*}
\omega \cdot y_0 &=& y_0 \\
\omega \cdot y_1 &=& y_1 + y_2 \\
\omega \cdot y_2 &=& y_1 
\end{eqnarray*}
and can be deduced from \cite[\S~2.5]{Bea1}. Therefore, one can completely calculate $H^{p}(\Sto; \E Z)$ without resorting to the duality resolution. However, most existing $K(2)$-local computations are done using the duality resolution spectral sequence, which is why we chose this method, providing a better basis for comparison with previous work.
\end{rem}
\begin{cor} The homotopy fixed point spectral sequence 
\[ E_2^{s,t}= H^s(\Sto; (E_2)_t Z) \To \pi_{t-s}(E^{h \Sto} \sma Z)\]
with $d_r: E_r^{s,t} \to E_r^{s+r, t+ r - 1}$, has $E_2$-page 
\[E_2^{s, *} = H^s(\Sto;\E Z) \iso \left\{\begin{array}{rl}\Ff[v_2^{\pm1}]\langle x_{0,0} \rangle &s=0\\\Ff[v_2^{\pm1}]\langle x_{1,0},x_{1,2},x_{1,4}\rangle&s=1\\\Ff[v_2^{\pm1}]\langle x_{2,2},x_{2,4},x_{2,6}\rangle&s=2\\\Ff[v_2^{\pm1}]\langle x_{3,0} \rangle &s=3\end{array}\right.\]or in graphical form (in Adams' grading) with each $\spadesuit$ denoting a copy of $\Ff[v_2^{\pm1}]$:
\[\begin{sseq}[grid=chess]{-1...4}{0...4}
\ssdrop{\spadesuit}
\ssmove{-1}{1}\ssdrop{\spadesuit}
\ssmove{2}{0}\ssdrop{\spadesuit}
\ssmove{2}{0}\ssdrop{\spadesuit}
\ssmove{-3}{1}\ssdrop{\spadesuit}
\ssmove{2}{0}\ssdrop{\spadesuit}
\ssmove{2}{0}\ssdrop{\spadesuit}
\ssmove{-1}{1}\ssdrop{\spadesuit}
\end{sseq}\]
The spectral sequence collapses at the $E_2$-page due to sparseness. 
\end{cor}
\begin{rem}
According to recent work of Goerss and Bobkova \cite{BG}, there is a topological version of the duality resolution, which gives a resolution of the $K(2)$-local sphere. The topological duality resolution can be used to compute $\pi_*( E_{2}^{h\Sto} \sma Z)$ directly. However, for $Z$, the algebraic and the topological duality spectral sequences are isomorphic and the computations remain identical as the relevant spectral sequences simply collapse.
\end{rem}

\section{The $K(2)$-local homotopy groups of $Z$} \label{Sec:homotopy}
The $K(2)$-local homotopy groups of $Z$ can be computed using the homotopy fixed point spectral sequence 
\begin{eqnarray} \label{descent}
E_2^{s,t} = H^{s}(\St; (E_2)_t Z)^{Gal(\Ff/\Ft)} \To \pi_{t-s}(L_{K(2)}Z),
\end{eqnarray} 
where $Gal(\Ff/\Ft)$ merely plays the role of `changing the coefficient field from $\Ff$ to $\Ft$.' 

Recall the norm map \eqref{norm}, $N: \St \to \Z_2$, whose kernel is $\Sto$. By choosing an element $\gamma \in \St$ such that $N(\gamma)$ is a topological generator of $\Z_2$, one can produce a map $\Z_2 \to Aut(\Sto)$ which sends $1\in\Z_2$ to the conjugation automorphism by $\gamma$, which gives an isomorphism
\[  \St \iso \Sto \rtimes \Z_2.\] 
In  \cite{Bea1, Bea2}, $\gamma$ is chosen to be  $\pi$. However, one can also choose $\gamma = \alpha^{-1} \pi$. We choose $\gamma = \alpha^{-1} \pi$ to get the isomorphism $\St \iso \Sto \rtimes \Z_2$.  This is convenient for us because $\alpha^{-1} \pi \in F_{4/2} \St$ and therefore it acts trivially on $\E Z$. Consequently, the Lyndon-Hochschild-Serre spectral sequence 
\[ H^{p}(\Z_2; H^{q}(\Sto; \E Z)) \To H^{p+q}(\St; \E Z)\]
collapses. Therefore 
\[ H^{*}(\St; \E Z)^{Gal(\Ff/\Ft)} \iso [E(\zeta) \otimes H^{*}(\Sto; \E Z)]^{Gal(\Ff/\Ft)}\]
where  $\zeta$ has bidegree $(s,t) = (1,0)$. More precisely, as a $\Ft[v_2^{\pm 1}]$-module  
\[
H^{s}(\St; \E Z)^{Gal(\Ff/\Ft)} \iso 
\left\{\begin{array}{rl}\Ft[v_2^{\pm1}]\langle x_{0,0} \rangle &s=0, \\
\Ft[v_2^{\pm1}]\langle \zeta x_{0,0}, x_{1,0},x_{1,2},x_{1,4} \rangle&s=1,\\
\Ft[v_2^{\pm1}]\langle x_{2,2},x_{2,4}, x_{2,6}, \zeta x_{1,0},\zeta x_{1,2},\zeta x_{1,4}\rangle& s=2, \\
\Ft[v_2^{\pm1}]\langle x_{3,0}, \zeta x_{2,2}, \zeta x_{2,4},\zeta  x_{2,6} \rangle & s=3,  \\
\Ft[v_2^{\pm1}]\langle \zeta x_{3,0} \rangle & s = 4, \\ 
0 & \text{otherwise}\end{array} \right.\]
\begin{figure}[!ht]
\[\begin{sseq}[grid=chess,entrysize= .64cm]{-8...10}{0...4}
\ssdropbull\ssname{g}\ssdroplabel[L]{x_{0,0}}
\ssmove{6}{0}\ssdropbull\ssname{v2}\ssdroplabel[L]{v_2x_{0,0}}
\ssmove{-12}{0}\ssdropbull\ssname{v2inv}
\ssmove{-1}{1}\ssdropbull\ssmove{2}{0}\ssdropbull  \ssmove{2}{0} \ssdropbull\ssname{v2invx4}\ssmove{2}{0}\ssdropbull\ssdroplabel[L]{x_{1,0}}\ssdrop{\circ}\ssdroplabel[L]{\zeta x_{0,0}}
\ssmove{2}{0}\ssdropbull \ssdroplabel[L]{x_{1,2}} \ssmove{2}{0}\ssdropbull\ssdroplabel[L]{x_{1,4}}\ssname{x4}\ssmove{2}{0}\ssdropbull\ssdrop{\circ}\ssmove{2}{0}\ssdropbull \ssmove{2}{0}\ssdropbull\ssname{v2x4}
\ssmove{-17}{1}\ssdropbull\ssdrop{\circ}\ssmove{2}{0}\ssdropbull\ssdrop{\circ}\ssmove{2}{0}\ssdropbull\ssdrop{\circ}\ssmove{2}{0}\ssdropbull\ssdrop{\circ}\ssmove{2}{0}\ssdropbull\ssdroplabel[L]{x_{2,2}}\ssdrop{\circ}\ssmove{2}{0}\ssdropbull\ssdroplabel[L]{x_{2,4}}\ssdrop{\circ}\ssmove{2}{0}\ssdropbull\ssdroplabel[L]{x_{2,6}}\ssdrop{\circ}\ssmove{2}{0}\ssdropbull\ssdrop{\circ}\ssmove{2}{0}\ssdropbull\ssdrop{\circ}\ssmove{2}{0}\ssdropbull\ssdrop{\circ}
\ssmove{-17}{1}\ssdrop{\circ}\ssname{d3v2inv}\ssmove{2}{0}\ssdrop{\circ}\ssmove{2}{0}\ssdropbull\ssdroplabel[R]{x_{3,0}}\ssdrop{\circ}\ssmove{2}{0}\ssdrop{\circ}\ssname{d3g}\ssmove{2}{0}\ssdrop{\circ}\ssmove{2}{0}\ssdropbull\ssdrop{\circ}\ssmove{2}{0}\ssdrop{\circ}\ssname{d3v2}\ssmove{2}{0}\ssdrop{\circ}\ssmove{2}{0}\ssdropbull\ssdrop{\circ}
\ssmove{-13}{1}\ssdrop{\circ}\ssname{d3v2invx4}\ssmove{6}{0}\ssdrop{\circ}\ssdroplabel[L]{v_2\zeta x_{3,0}}\ssname{d3x4}\ssmove{6}{0}\ssdrop{\circ}\ssname{d3v2x4}

\ssgoto{x4}\ssgoto{d3x4}\ssstroke[dashed,arrowto]
\ssgoto{v2x4}\ssgoto{d3v2x4}\ssstroke[dashed,arrowto]
\ssgoto{v2invx4}\ssgoto{d3v2invx4}\ssstroke[dashed,arrowto]
\end{sseq}\]
\caption{The spectral sequence $H^s(\Gt;\Et Z)\To\pi_{t-s}L_{K(2)}Z$}\label{LK2Z}
\end{figure}
In Figure~\ref{LK2Z}, we draw the $E_2$-page of \eqref{descent}. We denote by $\circ$ the generators that are multiples of $\zeta$, and all others by $\bullet$. 

It is clear that the spectral sequence \eqref{descent} collapses at the $E_4$-page. The only possibilities are two sets of $v_2$-linear $d_3$-differentials 
\begin{itemize}
\item $d_3(x_{0,0}) = v_2^{-1} \zeta x_{2,6}$, and,  
\item $d_3(x_{1,4}) = v_2 \zeta x_{3,0}$.
\end{itemize} 
The $v_2$-linearity of differentials follows from the fact that $Z$ admits a $v_2^1$-self-map \cite{BE}. However, the generator $x_{0,0}$ cannot support a differential for the following reason:

\noindent 
The inclusion  of the bottom $\iota_0: S^0 \hookrightarrow Z$ induces a nontrivial map $K(2)$-homology. Therefore, $\iota_0$ induces a nontrivial element in $\overline{\iota} \in \pi_0(L_{K(2)} Z)$ which is represented by $x_{0,0} \in $ in the $E_2$-page of  
the descent spectral sequence \ref{descent}. Therefore, $x_{0,0}$ is a permanent cycle.

\bigskip
 From the calculation of the classical Adams spectral sequence in \cite{BE} 
\[ Ext_{A}^{s,t}(H^*(Z), \Ft) \To \pi_{*}(Z)\]
we see that $\pi_0(Z) \iso \Z/2$. In particular, this means $[\iota_0]$ is the generator of $\pi_0(Z)$ and $2 [\iota_0] = 0$.  Since the map $\eta: Z \to L_{K(2)}Z$ sends $[\iota_0] \mapsto \overline{\iota}$, it must be the case that $2 \overline{\iota} =0$. Therefore there is no hidden extension supported by $x_{0,0}$.

 Moreover it is well known that $\tilde{\zeta}$ is a class in $\pi_{-1}L_{K(2)}S^0$. Let $\hat{\zeta}$ denote the representative of $\tilde{\zeta}$ in the $E_2$-page of the descent spectral sequence
 \[ E_2^{s,\ast}  = H^{s}(\Gt; \E S^0) \Rightarrow  \pi_{\ast - s}(L_{K(2)}S^0). \]
   A straightforward analysis of the map of descent spectral sequences induced by $\iota_0$ shows that $\hat{\zeta} \cdot x_{0,0} = \zeta x_{0,0} $, which is a nonzero permanent cycle representing $\tilde{\zeta} \cdot \overline{\iota} \in \pi_{-1}(L_{K(2)} Z)$. Since $2\overline{\iota} = 0$, it follows that $$2( \tilde{\zeta} \cdot  \overline{\iota}) = \tilde{\zeta}\cdot 2\overline{\iota} = 0,$$  ruling out another possible $v_2$-periodic family of hidden extensions. There are other possibilities of hidden extensions depicted in Figure~\ref{LK2Z1}, which we currently cannot rule out, though low dimensional computations lead us to believe that there exists a particular spectrum $Z$ for which all differentials and possible hidden extensions are zero. Furthermore, as stated in Conjecture \ref{collapse}, we expect that this will be the case for \emph{every} spectrum $Z \in \ZZ$.

\newpage
\appendix
\section{A regularity criterion for a representation of $Q_8$} \label{appendix}
The quarternionic group $Q_8$ is an order $8$ group which can be presented as
\begin{equation} \label{q8p}
 Q_8 = \langle \ii, \jj : \ii^4 = \ind, \ii^2 = \jj^2, \ii^3 \jj = \jj\ii \rangle  
 \end{equation}
We will denote the neutral element of $Q_8$ by $\ind$. Often in the literature, $\ii \jj$ is denoted by $\kk$ and $\ii^2$ by $-1$. This is justified as $-1\in Q_8$ is central and its square is $\ind$. However, $-1$ also denotes the additive inverse of $1$ in a ring, and potentially can cause confusion while working with group rings. Therefore we will instead denote $-1$ by $\nind\in Q_8$ and $\hi=\nind\ii,\hj=\nind\jj,\hk=\nind\kk$.
With this notation, the relations in $Q_8$ can be rewritten as
\begin{itemize}
	\item $\ii\jj = \kk$, $\jj\kk = \ii$, $\kk \ii = \jj$
	\item $\ii^2 = \jj^2 = \kk^2 = \nind $, 
	\item $(\nind)^2 = 1$, and
	\item $\jj\ii = \hk$, $\kk \jj = \hi$, $\ii \kk = \hj$.
\end{itemize}
The quotient of the central subgroup of order $2$ generated by $\nind$, is the \textit{Klein four group} $C_2 \times C_2$. In other words we have an exact sequence of groups 
\[ \ind \to C_2 \overset{\iota}\to Q_8 \overset{q}\to C_2\times C_2 \to \ind.\]
We will denote the images of $\ii,\jj\in Q_8$ by $\ii,\jj\in C_2\times C_2$.

Let $\FF$ be an arbitrary field and let $V_4(\FF)$ denote the $4$-dimensional representation of $Q_8$ induced by the regular representation of $C_2\times C_2$ via the quotient map $q$. Let $V_8(\FF)$ denote the regular representation of $Q_8$. When $\chr \FF = 2$, it is easy to see that there is an exact sequence of $\FF[Q_8]$-modules 
\[ 0 \to V_4(\FF) \overset{t}\to V_8(\FF) \overset{r}\to V_4(\FF) \to 0.\] 
More explicitly, let $\iota_4$ and $\iota_8$ be the generators of  $V_4(\FF)$ and $ V_8(\FF)$ as $\FF[Q_8]$-modules and define 
  \begin{eqnarray*}
  r(g \cdot \iota_8) &=& q(g) \cdot \iota_4 \\
  t(h \cdot \iota_4) &=& h \cdot \iota_8 + \nind h \cdot \iota_8,
\end{eqnarray*}
 for $h, g \in Q_8$.

The purpose of this appendix is to give a necessary and sufficient condition on an $8$-dimensional representation $V$ over a field $\FF$ with $\chr \FF = 2$, which fits in the exact sequence 
\begin{equation} \label{exact}
0 \to V_4(\FF) \overset{\tilde{t}}\to V \overset{\tilde{r}}\to V_4(\FF) \to 0,
\end{equation}
under which it is isomorphic to $V_8(\FF)$. When $\chr \FF \neq 2$, the problem is straightforward. Any $V$ which satisfies \eqref{exact} is isomorphic to $ V_4(\FF) \oplus V_4(\FF)$, including $V_8(\FF)$, the regular representation of $Q_8$. This is because, when $\chr \FF \nmid |Q_8|$, and $W$ is a subrepresentation of $V$, then one can define a complement subrepresentation $W'$ such that $V \iso W \oplus W'$ (Maschke's theorem). In our case, let $W = \img \tilde{t}$ and $W'$  be its complement. Since \eqref{exact} is an exact sequence, it follows that 
\[ W \iso W' \iso V_4(\FF).\]
We will soon see that $V_{8}(\FF) \ncong V_4(\FF) \oplus V_4(\FF)$ when $\chr \FF = 2$. 

For any $g \in G$, let $e_g\in \FF[G]$ denote the element such that\[g'e_g = e_{g'g}\] 
for every $g' \in G$. The collection $\set{e_g: g \in G}$ forms a basis for $\FF[G]$. For our convenience, we consider the ordered basis
\begin{equation}\label{appendix:B4} \mathcal{B}_4 = \set{ v_1 = e_{\ind}+ e_{\ii}+ e_{\jj}+e_{\kk} , v_2 = e_{\ind} + e_{\jj} , v_3 = e_{\ind} + e_{\ii} , v_4 = e_{\ind} }\end{equation}
of $V_4(\FF)$. Note that 
\[ (\ii)_ {\mathcal{B}_4}= \mtrx{1&1&0&0\\0&1&0&0\\ 0&0&1&1\\ 0&0&0&1} \]
\[ (\jj)_ {\mathcal{B}_4} = \mtrx{1&0&1&0\\0&1&0&1\\ 0&0&1&0\\ 0&0&0&1} \]
\[ (\kk)_ {\mathcal{B}_4}= \mtrx{1&1&1&1\\0&1&0&1\\ 0&0&1&1\\ 0&0&0&1}. \]
Thus any vector space isomorphic to the regular representation of $C_2 \times C_2$, admits a basis $\mathcal{B}$ such that 
\[ (\ii)_{\mathcal{B}} = (\ii)_{\mathcal{B}_4}, (\jj)_{\mathcal{B}} = (\jj)_{\mathcal{B}_4}, (\kk)_{\mathcal{B}} = (\kk)_{\mathcal{B}_4}.\]

The main result in this appendix is the following. 
\begin{lem} \label{appendixmain}Let $\FF$ be a field with $\chr \FF = 2$. Suppose we have an exact sequence of $\FF[Q_8]$-modules
\begin{equation} \label{exact1}
 0 \to V_4 \overset{\tilde{t}}\to V_8 \overset{\tilde{r}}{\to} V_4 \to 0
\end{equation}
where $V_4$ is a representation of $Q_8$ induced from the regular representation of $C_2 \times C_2$. Let $\mathcal{B}=\set{v_1,v_2,v_3,v_4}$ be a basis of $V_4$ such that 
\[ (\ii)_ {\mathcal{B}}= \mtrx{1&1&0&0\\0&1&0&0\\ 0&0&1&1\\ 0&0&0&1},  (\jj)_ {\mathcal{B}}= \mtrx{1&0&1&0\\0&1&0&1\\ 0&0&1&0\\ 0&0&0&1} \]
\[ (\kk)_ {\mathcal{B}}=\mtrx{1&1&1&1\\0&1&0&1\\ 0&0&1&1\\ 0&0&0&1}. \]
Then for any basis  $\mathcal{C}=\set{c_1 , c_2,c_3,c_4 , c_1', c_2', c_3', c_4'}$ of $V_8$ with the property that $\tilde{t}(v_i) = c_i$ and $\tilde{r}(c_i') = v_i$, we have 
\begin{enumerate}
\item\label{mainlemma:one}  $(\nind)_{\mathcal{C}} = \mtrx{I_4 & M \\ 0 & I_4}$, where 
\[M = \mtrx{c&d&a&b \\ 0& c&0& a \\ 0&0& c& d\\ 0&0&0&c} \]
for $a,b,c,d \in \FF$, and,
\item\label{mainlemma:two} if $c \neq 0$ then $V_8$ is isomorphic to the regular representation of $Q_8$.
\end{enumerate}
\end{lem}
\begin{proof}
It follows from \eqref{exact1} that 
\[ (\ii)_{\mathcal{C}} = \mtrx{(\ii)_{\mathcal{B}} & X \\ 0 & (\ii)_{\mathcal{B}} } \text{ and } (\jj)_{\mathcal{C}} = \mtrx{(\jj)_{\mathcal{B}}& Y \\ 0 &(\jj)_{\mathcal{B}}}  \]
for some $(4\times4)$ matrices $X, Y$. Let $x_{ij}$ and $y_{ij}$ denote the $(i,j)$-th entry of $X$ and $Y$ respectively.  Since the choice of $c_i'$ is only unique modulo $\img \tilde{t}$, we may apply a change of basis matrix of the form 
\[ P = \mtrx{ I_4 & \overline{P} \\ 0 & I_4 }.\]
In particular, if we choose 
\[ \overline{P}  = \mtrx{y_{13} & 0 & x_{14} & 0 \\
x_{11} & x_{12} + y_{13} & x_{13} & 0 \\
y_{11} & y_{12} & 0 & y_{14} \\
x_{31} & x_{32}+y_{11} & x_{33} & x_{34}
} \]
we see that 
\[ P (\ii)_{\mathcal{C}} P^{-1} = \mtrx{(\ii)_{\mathcal{B}}
 & \widetilde{X} \\ 0 & (\ii)_{\mathcal{B}} }, P (\jj)_{\mathcal{C}} P^{-1} = \mtrx{(\jj)_{\mathcal{B}} & \widetilde{Y} \\ 0 & (\jj)_{\mathcal{B}} }  \]
where 
\[ \tilde{X} = \mtrx{ 0&0&0&0 \\ x_{21} & x_{11} + x_{12} & x_{23} & x_{13} + x_{24} \\ 0&0&0&0 \\ x_{41} & x_{31}+x_{42} & x_{43} & x_{33} + x_{34}
} \]
and \[	
\tilde{Y} = \mtrx{
0 & 0& 0 & 0 \\
x_{31} + y_{21} & x_{32} + y_{11} + y_{22} & x_{11} + x_{33} + y_{23} & x_{12} + x_{34} + y_{13} + y_{24} \\ 
y_{31} & y_{32} & y_{11} + y_{33} & y_{12} + y_{34} \\ 
y_{41} & y_{42} & x_{31} + y_{43} & x_{32} + y_{11} + y_{44}	
}.
\]
Thus without loss of generality we may assume that 
\[ X = \mtrx{0 & 0 & 0 & 0 \\
x_{21} & x_{22} & x_{23} & x_{42} \\
0 & 0 & 0 &0 \\
x_{41} & x_{42} & x_{43} & x_{44} 
}
\text{ and }
Y = \mtrx{ 0 & 0 & 0 & 0 \\
y_{21} & y_{22} & y_{23} & y_{24} \\
y_{31} & y_{32} & y_{33} & y_{43} \\ 
y_{41} & y_{42} & y_{43} & y_{44}
}.
\]
Now we use the relations \eqref{q8p} to get further restrictions on $X$ and $Y$. While $(\ii)_{\mathcal{C}}^4 = (\jj)_{\mathcal{C}}^4 = I_8$ is trivially satisfied, $(\ii)_{\mathcal{C}}^2 = (\jj)_{\mathcal{C}}^2$ is true if and only if 
\[ (\ii)_{\mathcal{C}} X + X (\ii)_{\mathcal{C}} = (\jj)_{\mathcal{C}} Y + Y (\jj)_{\mathcal{C}}.  \]
Thus we get a linear system, which upon solving yields only $y_{23}, y_{24}, y_{32}, y_{33}, y_{34}, y_{42}, y_{43}, y_{44}$ as free variables and 
we get \[ 
X = \mtrx{
0&0&0&0 \\ 
y_{42}& y_{32} & y_{33} & y_{34} \\ 
0 & 0 &0 & 0 \\
0 & 0 & y_{42} & y_{32}
} \text{ and }  
Y = \mtrx{
0&0& 0&0 \\
y_{43} & y_{43}+ y_{44} & y_{23} & y_{24} \\ 
y_{42} & y_{32} & y_{33} & y_{34} \\
0 & y_{42} & y_{43} & y_{44}	
}
\]
Consequently, $(\nind)_{\mathcal{C}} = \mtrx{ I_4 & M \\ 0 & I_4 }$, where 
\[ M = \mtrx{
y_{42} & y_{32} & y_{33} & y_{34} \\ 
0 & y_{42} & 0 & y_{33} \\
0 & 0 & y_{42} & y_{32} \\ 
0 & 0 & 0 & y_{42}	
}\] 
Now, the linear system generated by the relation 
\[ (\ii)_{\mathcal{C}}(\jj)_{\mathcal{C}} = (\nind)_{\mathcal{C}}(\jj)_{\mathcal{C}}(\ii)_{\mathcal{C}} \]
has free variables $y_{33}, y_{34}, y_{43}, y_{44}$ and basic variables
\begin{eqnarray*}
y_{23} & =& y_{33} + y_{43} \\ 
y_{24} & =& y_{34} + y_{44} \\ 
y_{32} &=& y_{33} + y_{43} + y_{44} \\
y_{42} &=& y_{43}.
\end{eqnarray*}
Let $a= y_{33}, b= y_{34},c= y_{43}$ and $d= y_{33} + y_{43} + y_{44}$. In terms of $a, b, c, d$, we have 

\begin{equation} \label{XY}
 X = \mtrx{
0 & 0 & 0 &0 \\ 
c &  d & a & b \\ 
0 & 0 & 0 & 0 \\
0 & 0 &c & d
},
Y = \mtrx{0&0&0&0 \\
c & c+d & a+c & a+b+c+d \\ 
c & d & a& b \\ 
0 & c & c & a+c+d
}, 
\end{equation}
\begin{equation} \label{M}
M = \mtrx{
c &  d & a & b \\ 
0 & c & 0 & a \\ 
0 & 0 & c & d \\ 
0 & 0& 0 & c
}.
\end{equation} 
Recall that our change of basis matrix was of the form 
\[ P = \mtrx{ I_4 & \overline{P} \\ 0 & I_4},\] 
and thus $P^{-1} = P$ and we have
\[ P^{-1} (\nind)_{\mathcal{C}} P = \mtrx{ I_4 & M  \\ 0 & I_4 }\]
as $\chr \FF = 2$. This proves \eqref{mainlemma:one}. 

For \eqref{mainlemma:two}, we need to find a vector $\overline{v}$ such that 
\[ \set{g \overline{v}: g \in Q_8}\]
spans $V_8$. We choose $\overline{v} =c_4'= \mtrx{0 &0 &0 &0 &0 &0 &0 &1}^T$ in the basis $\mathcal{C}$. Let 
\[ A = \mtrx{ \overline{v} & (\nind)_{\mathcal{C}} \overline{v} & (\ii)_{\mathcal{C}} \overline{v} & (\hi)_{\mathcal{C}} \overline{v} & (\jj)_{\mathcal{C}} \overline{v}& (\hj)_{\mathcal{C}} \overline{v}& (\kk)_{\mathcal{C}} \overline{v} & (\hk)_{\mathcal{C}} \overline{v}  }.\]
Using \eqref{XY} and \eqref{M} we see that 
\[ A = \mtrx{
0&b&0& a+b&0 & b+d & a+b+c+d & 0 \\
0&a&b & a+b&a+b+c+d& b+d& a+c& 0 \\ 
0 &d&0& c+d&b& b+d & a+b+c+d& a+b\\
0 & c& d& c+d &a+c+d& a+d& a+c& a \\
0&0&0&0&0&0&1&1\\
0&0&0&0&1&1&1&1\\
0&0&1&1&0&0&1&1\\
1&1&1&1&1&1&1&1
}.
\]
By a tedious but straightforward calculation, we find
\[ \det A = c^4,\]
completing the proof of \eqref{mainlemma:two}.
\end{proof}

\begin{rem}
 When $\chr \FF =2 $, the representations $V_4(\FF) \oplus V_4(\FF)$ and $V_8(\FF)$ are not isomorphic. Without loss of generality we may assume $c= 1$ and $ a= b = d = 0$. Suppose there were an isomorphism between $V_4(\FF) \oplus V_4(\FF)$ and $V_8(\FF)$. Then there exists a invertible matrix $P$ such that 
\[ P \mtrx{(\nind)_{\mathcal{B}_4} & 0 \\
0 &  (\nind)_{\mathcal{B}_4}
} = (\nind)_{\mathcal{C}} P 
\]
Note that $(\nind)_{\mathcal{B}_4}$ is simply the identity matrix, while $(\nind)_{\mathcal{C}}$ is not. It follows easily that any matrix which satisfies the above condition is not invertible, hence a contradiction.
\end{rem}
\begin{rem}We are unaware of any classification theorem for $8$ dimensional representations of $Q_8$ over fields of characteristic $2$. We suspect that the question of how many isomorphism classes of $V$ satisfy \eqref{exact} can be resolved. A possible guess might be that there are overall $4$ isomorphism classes 
\begin{itemize}
	\item $c \neq 0$ (when $V \iso V_8(\FF)$),
	\item $c = 0$, $d\neq 0 $, 
	\item $c= 0$, $d = 0$, $a \neq 0$, 
	\item $c= 0$, $a = 0$, $d= 0$, $b \neq 0$, and
	\item $a= b= c = d = 0$ (when $V \iso V_4(\FF) \oplus V_4(\FF)$).
\end{itemize}
Since this is irrelevant to the purpose of the paper, we leave this question to the interested reader to verify.
\end{rem}


\end{document}